\documentclass[11pt]{scrartcl}

\usepackage{amsfonts,amsmath,amssymb,amsthm}
\usepackage{mathrsfs,mathtools}
\usepackage{enumerate}
\usepackage{hyperref}
\usepackage{esint}
\usepackage{graphicx}
\usepackage{bm}
\usepackage{commath}
\usepackage{esint}
\DeclareMathAlphabet{\mathpzc}{OT1}{pzc}{m}{it}

\usepackage[dvips,bottom=1.4in,right=1in,top=1in, left=1in]{geometry}

\newcommand{\abssec}[1]{\noindent\normalsize {\bfseries #1\quad }\ignorespaces}
\renewenvironment{abstract}{\abssec{Abstract}}{\par\vspace{.1in}}
\newenvironment{keywords}{\abssec{Key Words}}{\par\vspace{.1in}}
\newenvironment{AMSMOS}{\abssec{AMS subject
  classification}}{\par\vspace{.1in}}

\newtheorem{theorem}{Theorem}[section]
\newtheorem{proposition}[theorem]{Proposition}
\newtheorem{lemma}[theorem]{Lemma}

\newtheorem{definition}[theorem]{Definition}

\newtheorem{remark}[theorem]{Remark}
\newtheorem{assumption}[theorem]{Assumption}

\numberwithin{equation}{section}

\numberwithin{theorem}{section}

\newcommand\weak{\stackrel{\mathclap{\normalfont \ast}}{\rightharpoonup}}

\def\RR{{\mathbb{R}}}
\def\NN{{\mathbb{N}}}

\def\Om{\Omega}

\def\bOm{\overline{\Om}}
\def\pOm{\partial\Omega}

\usepackage[textsize=small]{todonotes}

\title{Optimal control of the coefficient for fractional and regional fractional {$p$}-{L}aplace equations: Approximation and convergence
\thanks{The work of the first author is partially supported by  NSF grant DMS-1521590.  The work of the second author is partially supported by the Air Force Office of Scientific Research under the Award No: FA9550-15-1-0027}}

\author{Harbir Antil\thanks{Department of Mathematical Sciences, George Mason University, Fairfax, VA 22030, USA. \texttt{hantil@gmu.edu}},
\and
Mahamadi Warma\thanks{University of Puerto Rico  (Rio Piedras Campus), College of Natural Sciences,
Department of Mathematics, PO Box 70377 San Juan PR
00936-8377, USA. \texttt{mahamadi.warma1@upr.edu, mjwarma@gmail.com}}
}

\begin{document}

\maketitle

  \begin{abstract}
In this paper we study optimal control problems with either fractional or regional fractional $p$-Laplace equation, of order $s$ and $p\in [2,\infty)$, as constraints over a bounded open set with Lipschitz continuous boundary. The control, which fulfills the pointwise box constraints, is given by the coefficient of the involved operator. 
To overcome the degeneracy of both fractional $p$-Laplacians, we introduce a regularization for both operators. We show existence and uniqueness of solution to the regularized state equations and existence of solution to the regularized optimal control problems. We also prove several auxiliary results for the regularized problems which are of independent interest. We conclude with the convergence of the regularized solutions. 
\end{abstract}


\begin{AMSMOS}
 35R11, 49J20, 49J45, 93C73.
\end{AMSMOS}

\section{Introduction}

Let $\Omega\subset\RR^N$ be a bounded open set with boundary $\partial\Omega$ and $p\in [2,\infty)$. In this paper we introduce and investigate the existence and approximation of solution to the following {\bf optimal control problem (OCP)}:
\begin{align}\label{Min}
\mbox{Minimize}\left\{ \mathbb I(\kappa,u):= {\frac 12} \int_{\Omega}|u-\xi|^2\;dx+\int_{\Omega}|\nabla \kappa|\right\},
\end{align}
subject to the {\bf state constraints} given by either the regional fractional $p$-Laplace equation
\begin{equation}\label{ellip-eq}
\begin{cases}
\mathcal L_{\Omega,p}^s(\kappa,u)+ u=f\;&\mbox{ in }\;\Omega\\
u=0&\mbox{ on }\;\pOm ,
\end{cases}
\end{equation}
or the fractional $p$-Laplace equation
\begin{equation}\label{e-frac-intro}
\begin{cases}
(-\Delta)_{p}^s(\kappa,u)+u=f\;\;&\mbox{ in }\;\Omega\\
u=0&\mbox{ on }\;\RR^N\setminus\Omega .
\end{cases}
\end{equation}
The control $\kappa$ fulfills the {\bf control constraints}
\begin{align}\label{ad-cont}
\kappa\in \mathfrak{A}_{ad}:=\Big\{\eta\in BV(\Omega):\;\xi_1(x)\le \eta(x)\le \xi_2(x)\;\mbox{ a.e. in }\;\Omega\Big\} .  
\end{align}
Here the regional fractional and fractional operators are given for $x\in\Om$ by:  
\begin{align}\label{op-L}
\mathcal L_{\Omega,p}^s(\kappa,u)(x)=C_{N,p,s}\mbox{P.V.}\int_{\Omega}\kappa(x-y)|u(x)-u(y)|^{p-2}\frac{u(x)-u(y)}{|x-y|^{N+sp}}\;dy,
\end{align}
and for $x\in\RR^N$ by
\begin{align}\label{op-LR}
(-\Delta)_{p}^s(\kappa,u(x)):=C_{N,p,s}\mbox{P.V.}\int_{\RR^N}\kappa(x-y)|u(x)-u(y)|^{p-2}\frac{u(x)-u(y)}{|x-y|^{N+sp}}\;dy ,
\end{align}
respectively.
Moreover, $\kappa:\RR^N\to[0,\infty)$ is a measurable and even function, that is,
\begin{align}\label{even}
\kappa(x)=\kappa(-x),\;\;\forall\;x\in\RR^N.
\end{align} 
In addition, $f$ is a given force and $\xi$ is the given data. The functions $\xi_1$ and $\xi_2$ in \eqref{ad-cont} are the {\bf control bounds} and fulfill $0< \alpha \le \xi_1(x) \le \xi_2(x)$, a.e. $x \in \Omega$, for some constant $\alpha>0$. The precise regularity requirements for these quantities and the domain $\Omega$ will be discussed in Section~\ref{main-result}. Notice that the {\bf control} $\kappa$ appears in the {\bf coefficient} of the quasilinear operators $\mathcal L_{\Omega,p}^s$ and $(-\Delta)_{p}^s$. For \eqref{e-frac-intro}, we let $0<s<1$. We restrict $s$ to $\frac12<s<1$ in the case \eqref{ellip-eq}, see Remark~\ref{rem:sle12} for more details.

Let $a\in L^\infty(\Omega)$ and set
\begin{align}\label{p-laplace2}
\Delta_{p,a}u:=\mbox{div}(a(x)|\nabla u|^{p-2}\nabla u).
\end{align}
Most recently, in \cite{CKL,KM} a similar optimal control problem as {\bf OCP} with $\mathcal L_{\Omega,p}^s$ replaced by $\Delta_{p,a}$ and the control $a(x)$ has been considered. 

Even though {\bf OCP} with the equation \eqref{ellip-eq} is a natural extension of \cite{CKL,KM}, however, the nonlocality of $\mathcal{L}_{\Omega,p}^s$ in comparison to the local operator $\Delta_{p,a}$ makes {\bf OCP} challenging. Indeed the papers \cite{GM,War-Pr}, where the authors considered $\kappa = 1$, realized that the standard techniques available for the local $p$-Laplace equation with the operator $\Delta_{p,a}$ are not directly applicable to the regional fractional $p$-Laplace equation \eqref{ellip-eq}. For the {\bf OCP} the additional complication occurs due to the fact that the operator $\mathcal{L}_{\Omega,p}^s$ may degenerate, see subsection~\ref{s:fpLap} for details. We also refer to \cite{CKL,KM} for a discussion related to this topic in case of $\Delta_{p,a}$. Similar complications can occur when the state constraints in {\bf OCP} are \eqref{e-frac-intro}.

The problem to search for coefficients in case of linear elliptic problems is classical, we  refer (but not limited) to \cite{Mur1,Mur2,ACRK,Tar} and their references. However, this is the first work which provides a mechanism to search for the coefficients in case of a quasilinear, possibly degenerate and fractional nonlocal problem. From a numerical point of view an added attraction of our theory is the fact that it is Hilbert space $L^2$-based instead of $L^p$-based theory. 

Subsequently, to tackle this degeneracy in the operators $\mathcal{L}^s_{\Omega,p}$ (and similarly to $(-\Delta)^s_p$), we introduce a {\bf regularized optimal control problem (ROCP)} and we conclude with the convergence of solution of the regularized problem. Notice that in this paper we discuss the convergence of the optimal controls. Due to the possible degeneracy in the state equation it is unclear how to derive the first order stationarity system for {\bf OCP}. However, {\bf ROCP} comes to rescue, indeed the latter is build to precisely avoid such degeneracy issues. In a forthcoming paper, we shall derive the limiting stationarity system corresponding to the first order stationarity for {\bf ROCP}.

Differential equations of fractional order have gained a lot of attraction in recent years due to the fact that
several phenomena in the sciences are more accurately modelled by such equations
rather than the traditional equations of integer order. Linear and nonlinear equations have been extensively studied. The
applications in industry are numerous and cover almost every area. From the long
list of phenomena which are more appropriately modelled by fractional differential
equations, we mention: viscoelasticity, anomalous transport and diffusion, hereditary
phenomena with long memory, nonlocal electrostatics, the latter being relevant to
drug design, and L\'evy motions which appear in important models in both applied
mathematics and applied probability, as well as in models in biology and ecology. We
refer to  \cite{Mai,PD,VS}  and their references for more details on this topic.

During the course of studying the {\bf OCP}, we show the well-posedness (existence, uniqueness, and continuous dependence on data) of our state equation \eqref{ellip-eq} and the regularized state equation \eqref{ellip-ep}. We further show several important results for the regularized state equation in subsection~\ref{s:reg_state_apriori}. Thus we not only address many challenging issues associated with the state equation \eqref{ellip-eq} but also initiate several new research directions with many possible extensions.

The rest of the paper is organized as follows: In section \ref{s:not_prelim} we introduce the function spaces needed to investigate our problem. We also provide a precise definition of the regional fractional $p$-Laplacian. The results in this section hold for any $0<s<1$. Hereafter, we assume that $\frac12<s<1$. We state our main results for {\bf OCP} with regional fractional $p$-Laplacian in section~\ref{main-result} which is followed by the introduction of the {\bf ROCP} in subsection~\ref{s:regp} and a statement of the convergence results. The well-posedness  of the state system is discussed in section~\ref{s:state}. Section~\ref{s:OCP_exist} discusses the existence of solution to {\bf OCP}. In section~\ref{s:regstate} we discuss well-posedness of the regularized state equation, which is followed by the existence of solution to {\bf ROCP} in section~\ref{s:regOCP_exist}. We show the convergence of {\bf ROCP} solutions to {\bf OCP} solutions in section~\ref{proof-23}. We conclude by studying {\bf OCP} with fractional equation \eqref{e-frac-intro} in section~\ref{s:ocpF}.

\section{Notation and Preliminaries}\label{s:not_prelim}
Here we introduce the function spaces needed to investigate our problem and also prove some intermediate results that will be used throughout the paper. The results stated in this section are valid for any $0<s<1$.

\subsection{The fractional order Sobolev spaces}\label{func-set}
In this (sub)section, we recall some well-known results on fractional order
Sobolev spaces that are needed throughout the article.

Let $\Omega \subset \mathbb{R}^{N}$ be an arbitrary bounded open set. 
For $p\in \lbrack
1,\infty )$ and $s\in (0,1)$, we denote by
\begin{equation*}
W^{s,p}(\Omega ):=\left\{ u\in L^{p}(\Omega ):\;\int_{\Omega }\int_{\Omega }%
\frac{|u(x)-u(y)|^{p}}{|x-y|^{N+ps}}dxdy<\infty \right\},
\end{equation*}
the fractional order Sobolev space endowed with the norm
\begin{equation*}
\Vert u\Vert _{W^{s,p}(\Omega )}:=\left( \int_{\Omega }|u|^{p}\;dx+\int_{\Omega }\int_{\Omega }\frac{|u(x)-u(y)|^{p}}{|x-y|^{N+ps}}%
dxdy\right) ^{\frac{1}{p}}.
\end{equation*}
We let
$\displaystyle
 W_{0}^{s,p}(\Omega ):=\overline{\mathcal{D}(\Omega )}^{W^{s,p}(\Omega )}$.
The following result is taken from \cite[Theorem 1.4.2.4, p.25]{Gris} (see also \cite{Chen,War-N}).

\begin{theorem}\label{theo-gris}
Let $\Omega\subset\RR^N$ be a bounded open set with a Lipschitz continuous boundary. Then the following assertions hold.
\begin{enumerate}
\item If $0<s\le\frac 1p$, then $W^{s,p}(\Omega )= W_0^{s,p}(\Omega)$. 
\item If $\frac 1p<s<1$, then $W_0^{s,p}(\Omega)$ is a proper closed subspace of $W^{s,p}(\Omega)$.
\end{enumerate}
\end{theorem}
It follows from Theorem \ref{theo-gris} that  for a bounded open set with a Lipschitz continuous boundary, if $\frac 1p<s<1$, then
\begin{equation}
\Vert u\Vert_{W_{0}^{s,p}(\Om)}=\left( 
\int_{\Omega}\int_{\Omega}\frac{|u(x)-u(y)|^{p}}{%
|x-y|^{N+sp}}dxdy\right) ^{\frac{1}{p}}  \label{norm-26}
\end{equation}%
defines an equivalent norm on $W_{0}^{s,p}(\Om)$. 
Let $p^{\star }$ be given by
\begin{equation}
p^{\star }=\frac{Np}{N-sp}\;\mbox{ if }\;N>sp\;\mbox{ and }\;p^{\star }\in
\lbrack p,\infty )\;\mbox{ if }\;N=sp.  \label{p-star}
\end{equation}
Then by \cite[Theorems~6.7 and 6.10]{NPV}, 
there exists a constant $C>0$ such that for every $u\in W_{0}^{s,p}(\Omega)$,
\begin{equation}
\Vert u\Vert _{L^q(\Omega) }\leq C\Vert u\Vert _{W_{0}^{s,p}(\Omega
)},\;\;\forall \;q\in \lbrack 1,p^{\star }].  \label{sob-emb}
\end{equation}%
 Moreover, the continuous embedding $W_{0}^{s,p}(\Omega)\hookrightarrow L^{q}(\Omega )$ is compact for every $q\in\lbrack 1,p^{\star })$ (see  e.g. \cite[Corollary~7.2]{NPV}).  If $N<sp$, then one has the continuous embedding $W_0^{s,p}(\Omega)\hookrightarrow C^{0,s-\frac Np}(\bOm)$ (see  e.g. \cite[Theorem 8.2]{NPV}).

We have the following.

\begin{proposition}\label{rem-sob-hol}
Let $\Omega\subset\RR^N$ be an arbitrary bounded open set and $p\in [1,\infty)$. Then the following assertions hold.
\begin{enumerate}
\item If $0<t\le s<1$, then $W_0^{s,p}(\Omega)\hookrightarrow W_0^{t,p}(\Omega)$.

\item For every $0<s<1$, we have that $W_0^{1,p}(\Omega)\hookrightarrow W_0^{s,p}(\Omega)$.

\item Let $q> p$. If $0<t<s<1$, then $W_0^{s,q}(\Omega)\hookrightarrow W_0^{t,p}(\Omega)$.
\end{enumerate}
\end{proposition}

\begin{proof}
The proof of the assertions (a), (b) and (c) is contained in \cite[Proposition 2.1]{NPV}, \cite[Proposition 2.3]{War-N} and in \cite[Proposition 1.2]{AW2017note}, respectively. 
\end{proof}

If $0<s<1$, $p\in (1,\infty )$ and $p^{\prime }:=\frac{p}{p-1}$, then  the space $
W^{-s,p^{\prime }}(\Omega )$ is defined as usual to be the dual of the
reflexive Banach space $W_{0}^{s,p}(\Om )$. For $u\in W^{s,p}(\Omega)$ we shall denote by $U_{(p,s)}$ the function defined on $\Omega\times\Omega$ by
\begin{align}\label{Up}
U_{(p,s)}(x,y):=\frac{u(x)-u(y)}{|x-y|^{\frac Np+s}}.
\end{align}

We will always denote by $\chi_E$ the characteristic function of  a set $E\subseteq\Omega\times\Omega$.

\begin{remark}\label{rem-1}
{\em 
Let $u\in W_0^{s,p}(\Om)$ and $\{u_n\}_{n\in\NN}$ a sequence in $W_0^{s,p}(\Om)$. Then the following assertions hold. 
\begin{enumerate}
\item  If $u_n$ converges weakly to $u$ in $W_0^{s,p}(\Om)$ as $n\to\infty$ (that is, $u_n\rightharpoonup u$ in $W_0^{s,p}(\Om)$ as $n\to\infty$), then for every  $\varphi\in \mathcal D(\Om)$,
\begin{align*}
&\lim_{n\to\infty}\int_{\Omega}\int_{\Omega}\frac{(u_n(x)-u_n(y))(\varphi(x)-\varphi(y))}{|x-y|^{N+2s}}\;dxdy\\
=&\int_{\Omega}\int_{\Omega}\frac{(u(x)-u(y))(\varphi(x)-\varphi(y))}{|x-y|^{N+2s}}\;dxdy.
\end{align*}

\item If $u_n\rightharpoonup u$ in $W_0^{s,p}(\Omega)$, then $U_{n,(p,s)}\rightharpoonup U_{(p,s)}$ in $L^p(\Omega\times\Omega)$ as $n\to\infty$.

\item If $u_n\rightharpoonup u$ in $W_0^{s,p}(\Omega)$ and $U_{n,(p,s)}\to U_{(p,s)}$ in $L^p(\Omega\times\Omega)$ as $n\to\infty$, then $u_n\to u$ in $W_0^{s,p}(\Omega)$.
\end{enumerate}
}
\end{remark}
For
more information on fractional order Sobolev spaces we refer the reader to
\cite{Chen,NPV,Gris,War-N} and the references therein.

\subsection{Functions of bounded variation}
\label{s:BV}
Let $\Omega\subset\RR^N$ be an arbitrary bounded open set. Let 
\begin{align*}
BV(\Omega):=\Big\{g\in L^1(\Omega):\; \|g\|_{BV(\Omega)}<\infty \Big\},
\end{align*}
be the space of functions of bounded variation,
where
\begin{align*}
\|g\|_{BV(\Omega)}:=\|g\|_{L^1(\Omega)}+\sup\left\{\int_{\Omega}g\;\mbox{div}(\Phi)\;dx:\;\Phi\in C_0^1(\Omega,\RR^N),\;|\Phi(x)|\le 1,\;x\in\Omega\right\}.
\end{align*}
For $g\in BV(\Omega)$, we denote by $\nabla g$ the distributional gradient of $g$. We notice that $\nabla g$ belongs to the space of Radon measures $\mathcal M(\Omega,\RR^N)$.

The following notion of convergence is contained in \cite[Definition 3.1]{AFP}.

\begin{remark}\label{rem-21}
{\em Let $g\in BV(\Omega)$ and $\{g_n\}_{n\in\NN}$ a sequence in $BV(\Omega)$. 
\begin{enumerate}
\item We say that $\{g_n\}_{n\in\NN}$ converges weakly$^\star$ ($\weak$) to $g\in BV(\Omega)$ as $n\to\infty$, if and only if the following two conditions hold.
\begin{itemize}
\item[(i)] $g_n\to g$ (strongly) in $L^1(\Omega)$ as $n\to\infty$, and
\item[(ii)] $\nabla g_n\weak \nabla g$ (weakly$^\star$) in $\mathcal M(\Omega,\RR^N)$ as $n\to\infty$, that is,
\begin{align*}
\lim_{n\to\infty}\int_{\Omega}\phi\;d\nabla g_n=\int_{\Omega}\phi\;d\nabla g,\;\;\forall\;\phi\in C_0(\Omega). 
\end{align*}
\end{itemize}
\item In addition, if $g_n$ converges strongly to some $\tilde g$ in $L^1(\Omega)$ as $n\to\infty$ and satisfies 
$\sup_{n\in\NN}\int_{\Omega}|\nabla g_n|<\infty$,
then
\begin{align*}
\tilde g\in BV(\Omega),\;\; \int_{\Omega}|\nabla \tilde g|\le \liminf_{n\to\infty}\int_{\Omega}|\nabla g_n|\;\mbox{ and }\; g_n\weak \tilde g\;\mbox{ in }\; BV(\Omega) \;\mbox{ as }\;n\to\infty.
\end{align*}
\end{enumerate}
}
\end{remark}
For more details on functions of bounded variation we refer to \cite[Chapter 3]{AFP}.

\subsection{The regional fractional $p$-Laplacian}\label{s:fpLap}

Let $0<s<1$ and $p\in (1,\infty )$. The {\bf regional fractional $p$-Laplacian} $%
(-\Delta )_{\Om,p}^{s}$ is defined for  $x\in \Omega$ by the formula
\begin{equation*}
(-\Delta )_{\Om,p}^{s}u(x)=C_{N,p,s}\mbox{P.V.}\int_{\Omega}|u(x)-u(y)|^{p-2}\frac{u(x)-u(y)}{|x-y|^{N+ps}}dy,
\end{equation*}%
where $C_{N,p,s}$ is a normalized constant (see, e.g., \cite%
{BCF,Caf3,NPV,War-DN1} for the linear case $p=2,$ and \cite{War-IP,War-Pr,War-DN}
for the general case $p\in (1,\infty )$). We notice that if $0<s<\frac{p-1}{p}$
and $u$ is smooth (i.e., at least bounded and Lipschitz continuous on $\Omega$), 
then the above integral is in fact not really singular near $x$ (see e.g. \cite[Section 2.1]{AW2017note} for more details).
 If $\Omega=\RR^N$, then $(-\Delta )_{\RR^N,p}^{s}=(-\Delta )_{p}^{s}$ is usually called the {\bf fractional $p$-Laplace operator}, see section~\ref{s:ocpF} for more details.

It has been shown in \cite[Formula (2.4)]{AW2017note} that for every $u\in\mathcal D(\Omega)$,
\begin{align}\label{con-sp-p}
\lim_{s\uparrow 1}\int_{\Omega}u(-\Delta)_{\Omega,p}^su\;dx
=\lim_{s\uparrow 1}\int_{\RR^N}u(-\Delta)_{p}^su\;dx=\int_{\Omega}|\nabla u|^p\;dx=-\int_{\Omega}u\Delta_pu\;dx.
\end{align}
It follows from \eqref{con-sp-p} that the regional fractional $p$-Laplace operator converges (in some sense) to the $p$-Laplace operator, as $s\uparrow 1$.  

Let $\kappa$ be as in \eqref{even}. For $1<p<\infty$ and $0<s<1$ 
we define the operator $\mathcal L_{\Omega,p}^s$ as in \eqref{op-L}. We again call this operator, the {\bf regional fractional $p$-Laplace operator}.
 We mention that elliptic problems associated with the operator $\mathcal L_{\Omega,p}^s(\kappa,\cdot)$ subject to the Dirichlet boundary condition have been investigated in \cite{DiKG,DiKG2,GM,KMS} where the authors have obtained some fundamental existence and regularity results. The case of Neumann and Robin type boundary conditions (with $\kappa=1$) is contained in \cite{War-DN}. We refer to \cite{GM,War-Pr} for further results on parabolic problems.

\section{The main results}\label{main-result}

In this section we state the main results of the article. 
Throughout the remainder of the article, unless stated otherwise,
we assume the following.

\begin{assumption}\label{asump}
We shall always assume the following.
\begin{enumerate}
\item $\Omega\subset\RR^N$ {\em($N\ge 1$)} is a bounded open set with Lipschitz continuous boundary. 

\item $\frac 12<s<1$.
 
 \item The functions $\xi_1,\xi_2\in L^\infty(\Omega)$ and there exists a constant $\alpha>0$ such that
\begin{align}\label{func-xi}
0<\alpha\le \xi_1(x)\le\xi_2(x)\;\mbox{ a.e. in }\;\Omega.
\end{align}

\item The measurable function $\kappa$ satisfies the assumption given in \eqref{even}.
\end{enumerate}
\end{assumption}

Recall that it follows from Assumption \ref{asump}(b) that \eqref{norm-26} defines an equivalent norm on $W_0^{s,p}(\Omega)$ for every $p\in [2,\infty)$. 

\begin{remark}[Regional case: $0<s<1$]\label{rem:sle12}
\rm
The restriction on $s$ in Assumption \ref{asump}(b) is used to show the uniqueness of solution to \eqref{ellip-eq}. This step requires the equivalence between $\|v\|_{W^{s,p}_0(\Omega)}$ and $\left(\|v\|^2_{L^2(\Omega)} + \int_{\Omega}\int_{\Omega}\frac{|v(x)-v(y)|^{p}}{
|x-y|^{N+sp}}dxdy \right)^\frac1p$.  The equivalence follows immediately when $\frac12<s<1$. When $0<s\le\frac12$ the proof will be along the line of \cite[Corollary 1.5.2 p.37]{MP} (they only discuss $s=1$) but we do not explore this here. 
\end{remark}

\subsection{The optimal control problem}

Let $\xi,f\in L^2(\Omega)$ be given functions. 
The {\bf OCP} we consider first is \eqref{Min}, \eqref{ellip-eq} and \eqref{ad-cont}.
The following is our notion of solutions to the state system \eqref{ellip-eq}.

\begin{definition}
A function $u\in W_0^{s,p}(\Omega)$ is said to be a weak solution of the system \eqref{ellip-eq} if for every $\varphi\in W_0^{s,p}(\Omega)$,
\begin{align}\label{eq-15}
&\frac{C_{N,p,s}}{2}\int_{\Omega}\int_{\Omega}\kappa(x-y)|u(x)-u(y)|^{p-2}\frac{(u(x)-u(y))(\varphi(x)-\varphi(y))}{|x-y|^{N+sp}}\;dxdy\notag\\
& +\int_{\Omega}u\varphi\;dx=\int_{\Omega}f\varphi\;dx.
\end{align}
\end{definition}

The following existence result of optimal pair to the {\bf OCP} is our first main result. 

\begin{theorem}\label{theo-24}
Let $\xi,f\in L^2(\Omega)$ be given. Then the {\bf OCP} \eqref{Min}, \eqref{ellip-eq} and \eqref{ad-cont} admits at least one solution $(\kappa,u)\in BV(\Omega)\times W_0^{s,p}(\Omega)$.
\end{theorem}

\subsection{The regularized optimal control problem}\label{s:regp}
Let $\xi,f\in L^2(\Omega)$ be given functions and $p\in [2,\infty)$. 
 Let $n\in\NN$ and $\mathcal F_n:[0,\infty)\to [0,\infty)$ be a function in $C^1([0,\infty))$ satisfying
 \begin{equation}\label{eq-F}
 \begin{cases}
 \mathcal F_n(\tau)= \tau\;\;\;&\mbox{ if }\;  0\le\tau\le n^2,\\
 \mathcal F_n(\tau)= n^2+1&\mbox{ if }\; \tau>n^2+1,\\
 \tau\le \mathcal F_n(\tau)\le \tau+\delta\;\;&\mbox{ if }\; n^2\le \tau <n^2+1\;\mbox{ for some }\; \delta\in (0,1).
 \end{cases}
 \end{equation}
 Let $\varepsilon>0$ be a small parameter. The operator $\Delta_{p,a}$ defined in \eqref{p-laplace2} is degenerate if $p>2$. To overcome the degeneracy, an $(\varepsilon,p)$-regularization $\Delta_{\varepsilon,n, p,a}$ of $\Delta_{p,a}$ has been introduced (see e.g. \cite{CKL}) as follows:
 \begin{align*}
 \Delta_{\varepsilon,n, p,a}u=\mbox{div}\Big(a(x)(\varepsilon+\mathcal F_n(|\nabla u|^2))^{\frac{p-2}{2}}\nabla u\Big),
 \end{align*}
 where $\mathcal F_n$ is the function defined in \eqref{eq-F}. Using the classical definition of degenerate elliptic operators, one cannot immediately say that $(-\Delta)_{\Omega,p}^s$ or $\mathcal L_{\Om,p}^s(\kappa,\cdot)$ is degenerate for $p>2$. We refer to \cite{Vaz} for a discussion on this topic.  

But inspired by the convergence given in \eqref{con-sp-p}, we let
 \begin{align}
 &\mathcal L_{\Om,p,\varepsilon,n}^s(\kappa,u)(x):=\notag\\
 C_{N,p,s}\mbox{P.V.}&\int_{\Omega}\kappa(x-y)\left[\varepsilon+\mathcal F_n\left(\frac{|u(x)-u(y)|^2}{|x-y|^{2s}}\right)\right]^{\frac{p-2}{2}}\frac{u(x)-u(y)}{|x-y|^{N+2s}}\;dy.
 \end{align} 
 We call $\mathcal L_{\Om,p,\varepsilon,n}^s$ an {\bf $(\varepsilon,p)$-regularization}  of $\mathcal L_{\Om,p}^s$. 
 
 Now we consider our so called regularized optimal control problem {\bf(ROCP)}:
 
 \begin{align}\label{eq-27}
 \mbox{Minimize}\left\{ \mathbb I(\kappa,u):= {\frac 12} \int_{\Omega}|u-\xi|^2\;dx+\int_{\Omega}|\nabla \kappa|\right\}
 \end{align}
subject to the constraints
\begin{align}\label{eq-29}
\kappa\in \mathfrak{A}_{ad}=\Big\{\eta \in BV(\Omega):\; \xi_1(x)\le \eta(x)\le\xi_2(x)\;\mbox{ a.e. in }\Omega\Big\},
\end{align}
and
\begin{equation}\label{ellip-ep}
\begin{cases}
 \mathcal L_{\Om,p,\varepsilon,n}^s(\kappa,u)+ u=f\;\;&\mbox{ in }\;\Omega,\\ 
 u=0&\mbox{ on }\;\pOm.
 \end{cases}
\end{equation}

The following is our notion of weak solution to the system \eqref{ellip-ep}.

\begin{definition}
Let $n\in\NN$, $\varepsilon>0$, $\kappa\in\mathfrak{A}_{ad}$ and $f\in L^2(\Omega)$. A function $u\in W_0^{s,2}(\Omega)$ is said to be a weak solution to the system \eqref{ellip-ep} if the equality
\begin{align}\label{solu-epsi}
\mathbb F_{\varepsilon,n,p}^\kappa(u,v)=\int_{\Omega}fv\;dx
\end{align}
holds for every $v\in W_0^{s,2}(\Omega)$, where we have set
\begin{align}\label{func-FF}
&\mathbb F_{\varepsilon,n,p}^\kappa(u,v):=\int_{\Omega}uv\;dx\\
+&\frac{C_{N,p,s}}{2}\int_{\Omega}\int_{\Omega}\kappa(x-y)\Big[\varepsilon+\mathcal G_n(u,s)\Big]^{\frac{p-2}{2}}\frac{(u(x)-u(y))(v(x)-v(y))}{|x-y|^{N+2s}}\;dxdy ,\notag
\end{align}
and
\begin{align}\label{FUNC-G}
\mathcal G_n(u,s):=\mathcal F_n\left(\frac{|u(x)-u(y)|^2}{|x-y|^{2s}}\right).
\end{align}
\end{definition}

The following theorem is our second main result.

\begin{theorem}\label{theo-ex-ep}
For every $\varepsilon>0$ and $n\in\NN$, the {\bf ROCP} \eqref{eq-27}-\eqref{ellip-ep} has at least one solution $(\kappa_{\varepsilon,n},u_{\varepsilon,n})\in BV(\Omega)\times W_0^{s,2}(\Omega)$.
\end{theorem}

We conclude this section by stating the convergence of solutions of the {\bf ROCP} to the solutions of the {\bf OCP}.  

\begin{theorem}\label{theorem-23}
Let $\frac 12<t\le  s<1$ and $p\in [2,\infty)$ with $t=s$ if $p=2$. Let $n\in\NN$ and $\varepsilon>0$.
Let $\{(\kappa_{\varepsilon,n}^\star,u_{\varepsilon,n}^\star)\}_{\varepsilon>0,n\in\NN}\subset BV(\Omega)\times W_0^{s,2}(\Omega)$ be an arbitrary sequence of solutions to the {\bf ROCP} \eqref{eq-27}-\eqref{ellip-ep}. 
Then $\{(\kappa_{\varepsilon,n}^\star,u_{\varepsilon,n}^\star)\}_{\varepsilon>0,n\in\NN}$ is bounded in $BV(\Omega)\times W_0^{t,2}(\Omega)$ and any cluster point $(\kappa_\star,u_\star)$ with respect to the {\em (weak$^\star$, weak)} topology of $BV(\Omega)\times W_0^{t,2}(\Omega)$ is a solution to the {\bf OCP} \eqref{Min}, \eqref{ellip-eq} and \eqref{ad-cont}. 
In addition, if $\kappa_{\varepsilon,n}^\star\weak \kappa_\star$ in $BV(\Omega)$ and $u_{\varepsilon,n}^\star\rightharpoonup u_\star$ in $W_0^{t,2}(\Omega)$, as $\varepsilon\to 0$ and $n\to\infty$ {\em (that is, as $(\varepsilon,n)\to(0,\infty)$)}, then the following assertions hold.
\begin{align}
&\lim_{(\varepsilon,n)\to (0,\infty)}(\kappa_{\varepsilon,n}^\star,u_{\varepsilon,n}^\star) =(\kappa_\star,u_\star)\;\mbox{ stongly in }\; L^1(\Omega)\times W_0^{t,2}(\Omega).\label{210}\\
&\lim_{(\varepsilon,n)\to (0,\infty)}\int_{\Omega}|\nabla \kappa_{\varepsilon,n}^\star|=\int_{\Omega}|\nabla \kappa_\star|.\label{211}\\
&\lim_{(\varepsilon,n)\to (0,\infty)}\chi_{(\Omega\times\Omega)\setminus(\Omega\times\Omega)_n(u_{\varepsilon,n}^\star)}U_{\varepsilon,n,(p,s)}^\star=U_{\star,(p,s)}\;\mbox{ stongly in }\; L^p(\Omega\times\Omega).\label{212}\\
&\lim_{(\varepsilon,n)\to (0,\infty)}\int_{\Omega}\int_{\Omega}\kappa_{\varepsilon,n}^\star(x-y)\left[\varepsilon+\mathcal G_n\left(u_{\varepsilon,n}^\star,s\right)\right]^{\frac{p-2}{2}}\frac{|u_{\varepsilon,n}^\star(x)-u_{\varepsilon,n}^\star(y)|^2}{|x-y|^{N+2s}}\;dy\notag\\
&\qquad\qquad\qquad\qquad=\int_{\Omega}\int_{\Omega}\frac{|u_\star(x)-u_\star(y)|^p}{|x-y|^{N+sp}}\;dxdy.\label{213}\\
&\lim_{(\varepsilon,n)\to (0,\infty)} \mathbb I(\kappa_{\varepsilon,n}^\star,u_{\varepsilon,n}^\star)=\mathbb I(\kappa_\star,u_\star),\label{214}
\end{align}
where we recall that $\mathcal G_n$ is given by \eqref{FUNC-G}.
\end{theorem}

\section{Proof of Theorem \ref{theo-24}}\label{proof-24}

To prove the first main result we need some preparations and some intermediate important results. 

\subsection{The state equation is well-posed}\label{s:state}

Throughout the remainder of the paper for $u,\varphi\in W_0^{s,p}(\Omega)$, we shall let
\begin{align}\label{form}
&\mathcal E_{p,s}^\kappa(u,\varphi):=\int_{\Omega}u\varphi\;dx\notag\\
&+\frac{C_{N,p,s}}{2}\int_{\Omega}\int_{\Omega}\kappa(x-y)|u(x)-u(y)|^{p-2}\frac{(u(x)-u(y))(\varphi(x)-\varphi(y))}{|x-y|^{N+sp}}\;dxdy.
\end{align}
We have the following result of existence of weak solutions to the system \eqref{ellip-eq}.

\begin{proposition}[\bf The well-posedness of the state equation]\label{prop-ex-sol}
For every $f\in L^2(\Omega)$, the system \eqref{ellip-eq} has a unique weak solution $u\in W_0^{s,p}(\Omega)$. In addition there exists a constant $C>0$  such that
\begin{align}\label{sol-esti}
\|u\|_{W_0^{s,p}(\Omega)}^{p-1}\le C\|f\|_{L^2(\Omega)}.
\end{align}
\end{proposition}

\begin{proof}
The proposition follows by showing first that $\mathcal E_{p,s}^\kappa(u,\cdot)\in W^{-s,p'}(\Omega)$ for every fixed $u\in W_0^{s,p}(\Omega)$, and then that $\mathcal E_{p,s}^\kappa$ is strictly monotone, hemi-continuous and coercive.
Finally \eqref{sol-esti} follows by taking $\varphi=u$ as a test function in \eqref{eq-15}. For more details we refer to \cite[Proposition 2.3]{AW2017note}. The proof is finished. 
\end{proof}

\begin{remark}[\bf The state equation and Minty relation]\label{rem}
{\em 
As a consequence of the proof of Proposition \ref{prop-ex-sol}, we have that $u\in W_0^{s,p}(\Omega)$ satisfies \eqref{eq-15} if and only if the {\bf Minty relation} holds. That is, for every $\varphi\in W_0^{s,p}(\Omega)$,
\begin{align}\label{eq-minty}
\mathcal E_{p,s}^\kappa(\varphi ,\varphi-u)\ge \int_{\Omega}f(\varphi-u)\;dx.
\end{align}
For more details we refer to \cite[Remark 2.5]{AW2017note}.}
\end{remark}

\subsection{The optimal control problem (OCP)}\label{s:OCP_exist}

Towards this end we introduce the set of admissible control-state pair for the {\bf OCP} \eqref{Min}-\eqref{ellip-eq}, namely,
\begin{align}\label{set-E}
\Xi:=\Big\{(\kappa,u):\; \kappa\in \mathfrak{A}_{ad},\;u\in W_0^{s,p}(\Omega),\; (\kappa,u) \mbox{ are related by } \eqref{eq-15}\Big\}.
\end{align}
Using Proposition \ref{prop-ex-sol}, we get that the set $\Xi$
is nonempty. With the notation \eqref{set-E}, we have that the  {\bf OCP} \eqref{Min}-\eqref{ellip-eq} can be rewritten as the following minimization problem:
\begin{align}\label{eq-min}
\min_{(\kappa,u)\in\Xi} \mathbb I(\kappa,u).
\end{align}

Next, we endow the Banach space $BV(\Omega)\times W_0^{s,p}(\Omega)$ with the norm defined by
\begin{align*}
\|(\kappa,u)\|_{BV(\Omega)\times W_0^{s,p}(\Omega)}:=\|\kappa\|_{BV(\Omega)}+\|u\|_{W_0^{s,p}(\Omega)}.
\end{align*}

We have the following result.

\begin{lemma}\label{lem-con}
Let $\{(\kappa_n,u_n)\}_{n\in\NN}\subset\Xi$ be such that $\kappa_n\weak \kappa$ in $BV(\Omega)$ and $u_n\rightharpoonup u$ in $W_0^{s,p}(\Omega)$, as $n\to\infty$. Then for every $\varphi\in\mathcal D(\Omega)$ we have that
\begin{align}\label{eq22}
&\lim_{n\to\infty}\int_{\Omega}\int_{\Omega}\kappa_n(x-y)\frac{(u_n(x)-u_n(y))(\varphi(x)-\varphi(y))}{|x-y|^{N+2s}}\;dxdy\notag\\
=&\int_{\Omega}\int_{\Omega}\kappa(x-y)\frac{(u(x)-u(y))(\varphi(x)-\varphi(y))}{|x-y|^{N+2s}}\;dxdy.
\end{align}
\end{lemma}

\begin{proof}
First, since $\kappa_n\to \kappa$  in $L^1(\Omega)$ as $n\to\infty$ and $\{\kappa_n\}_{n\in\NN}$ is bounded in $L^\infty(\Omega)$, we have that 
\begin{equation}\label{EQ}
\kappa_n\to \kappa\; \mbox{ in }\; L^q(\Omega) \;\mbox{ as }\;n\to\infty, \mbox{ for every }\; 1\le q<\infty.
\end{equation}
 In addition we have that $\kappa\in\mathfrak A_{ad}$.
Since $u_n\rightharpoonup u$ in $W_0^{s,p}(\Omega)$ as $n\to\infty$, it follows from Remark \ref{rem-1} that for every $\varphi\in\mathcal D(\Omega)$,
\begin{align*}
&\lim_{n\to\infty}\int_{\Omega}\int_{\Omega}\frac{(u_n(x)-u_n(y))(\varphi(x)-\varphi(y))}{|x-y|^{N+2s}}\;dxdy \\
=&\int_{\Omega}\int_{\Omega}\frac{(u(x)-u(y))(\varphi(x)-\varphi(y))}{|x-y|^{N+2s}}\;dxdy.
\end{align*}
Let $\varphi\in\mathcal D(\Omega)$ and define the functions $F_{n,p}^\varphi, F_p^\varphi:\Omega\times\Omega\to\RR$ by
\begin{align*}
F_{n,p}^\varphi(x,y):=\frac{(u_n(x)-u_n(y))(\varphi(x)-\varphi(y))}{|x-y|^{\frac Np+s+1}}
\end{align*}
and
\begin{align*}
F_p^\varphi(x,y):=\frac{(u(x)-u(y))(\varphi(x)-\varphi(y))}{|x-y|^{\frac Np+s+1}}.
\end{align*}
Then
\begin{align}\label{JW1}
\int_{\Omega}\int_{\Omega}|F_{n,p}^\varphi(x,y)|^p\;dxdy
\le &\|\varphi\|_{C^{0,1}(\bOm)}^p\|u_n\|_{W_0^{s,p}(\Omega)}^p<\infty.
\end{align}
Similarly, we get that $F_{n,p}^\varphi, F_{p}^\varphi\in L^p(\Om\times\Om)$. Since $\{u_n\}_{n\in\NN}$ is bounded in $W_0^{s,p}(\Omega)$, it follows from \eqref{JW1} that $\{F_{n,p}^\varphi\}_{n\in\NN}$ is bounded in $L^p(\Om\times\Om)$. Thus, after a (sub)sequence if necessary,  $F_{n,p}^\varphi$ converges weakly to some function $F$ in $L^p(\Omega\times\Omega)$, as $n\to\infty$. Since $u_n$ converges a.e. to $u$ in $\Omega$ as $n\to\infty$, it follows that $F_{n,p}^\varphi$ converges a.e. to $F_p^\varphi$ in $\Omega\times\Omega$, as $n\to\infty$. By the uniqueness of the limit we have that $F_p^\varphi= F$. 
We have shown that $F_{n,p}^\varphi\rightharpoonup F_p^\varphi$ in $L^p(\Omega\times\Omega)$ as $n\to\infty$.
Let $K_{n,p'}, K_{p'}:\Omega\times\Omega\to\RR$ be the functions given by
\begin{align*}
K_{n,p'}(x,y):=\frac{\kappa_n(x-y)}{|x-y|^{\frac{N}{p'}+s-1}}\;\mbox{ and } K_{p'}(x,y):=\frac{\kappa(x-y)}{|x-y|^{\frac{N}{p'}+s-1}}.
\end{align*}
Let $x\in\Omega$ be fixed. Let $B(x,R)$ be a large ball with center $x$ and radius $R$ such that $\Omega\subset B(x,R)$. Since $\kappa_n\in L^\infty(\Omega)$, then using polar coordinates, we have that there exists a constant $C>0$ (depending on $\Omega$, $N$, $s$ and $p$) such that
\begin{align*}
\int_{\Omega}\int_{\Omega}|K_{n,p'}(x,y)|^{p'}\;dxdy\le &\|\kappa_n\|_{L^\infty(\Omega)}^{p'}\int_{\Omega}\int_{\Omega}\frac{1}{|x-y|^{N+p'(s-1)}}\;dxdy\\
\le &C\|\kappa_n\|_{L^\infty(\Omega)}^{p'}\int _0^Rr^{p'(1-s)-1}\;dr\le C\|\xi_2\|_{L^\infty(\Omega)}^{p'}<\infty.
\end{align*}
Thus $K_{n,p'}\in L^{p'}(\Omega\times\Omega)$. Similarly, we get that $K_{p'}\in L^{p'}(\Omega\times\Omega)$.
Using \eqref{EQ} and the Lebesgue Dominated Convergence Theorem, we get that $K_{n,p'}\to K_{p'}$ in $L^{p'}(\Omega\times\Omega)$ as $n\to\infty$. Therefore,
\begin{align*}
&\lim_{n\to\infty}\int_{\Omega}\int_{\Omega}K_{n,p'}(x,y)F_{n,p}^\varphi(x,y)\;dxdy
=\int_{\Omega}\int_{\Omega}K_{p'}(x,y)F_p^\varphi(x,y)\;dxdy\\
=&\int_{\Omega}\int_{\Omega}\kappa(x-y)\frac{(u(x)-u(y))(\varphi(x)-\varphi(y))}{|x-y|^{N+2s}}\;dxdy,
\end{align*}
for every $\varphi\in \mathcal D(\Omega)$.
We have shown \eqref{eq22} and the proof is finished.
\end{proof}

Using Lemma \ref{lem-con} we can prove the following theorem which will play an important role in the proof of our first main result.

\begin{theorem}\label{Theo-con}
Let $\{(\kappa_n,u_n)\}_{n\in\NN}\subset\Xi$ be a bounded sequence. Then there exists $(\kappa,u)\in\Xi$ such that, after a {\em (sub)}sequence if necessary, $\kappa_n\weak\kappa$ in $BV(\Omega)$, $u_n\rightharpoonup u$ in $W_0^{s,p}(\Omega)$  and $u_n\to u$ in $L^2(\Omega)$, as $n\to\infty$.
\end{theorem}

\begin{proof}
First, since $\{u_n\}_{n\in\NN}$ is bounded in $W_0^{s,p}(\Omega)$ and the continuous embedding $W_0^{s,p}(\Omega)\hookrightarrow L^2(\Omega)$ is compact, then after a (sub)sequence if necessary, there exists a $u\in W_0^{s,p}(\Omega)$ such that
\begin{align*}
u_n\rightharpoonup u\;\mbox{ in }\; W_0^{s,p}(\Omega) \;\mbox{ and }\; u_n\to u\;\mbox{ in }\;L^2(\Omega),\;\mbox{ as }\;n\to\infty.
\end{align*}
Next, since $\{\kappa_n\}_{n\in\NN}$ is bounded in $BV(\Omega)$, it follows from \cite[Corollary 3.39]{AFP} that after a (sub)sequence if necessary, there exists a $\kappa\in L^1(\Omega)$  such that
$\displaystyle
\kappa_n\to \kappa\;\mbox{ in }\; L^1(\Omega)$.
Since $\kappa_n\to \kappa$ in $L^1(\Omega)$ as $n\to\infty$ and $\sup_{n\in\NN}\int_{\Omega}|\nabla\kappa_n|<\infty$ (this follows from the fact that $\{\kappa_n\}_{n\in\NN}$ is bounded in $BV(\Omega)$), then by Remark \ref{rem-21}(b), this implies that $\kappa\in BV(\Omega)$ and
$\displaystyle
\kappa_n\weak\kappa\;\mbox{ in }\; BV(\Omega)\;\mbox{ as }\;n\to\infty$.
We have shown that, as $n\to\infty$,
\begin{align}\label{eq-33}
\kappa_n\weak\kappa\;\mbox{ in }\; BV(\Omega),\; u_n\rightharpoonup u\;\mbox{ in }\; W_0^{s,p}(\Omega) \;\mbox{ and }\; u_n\to u\;\mbox{ in }\;L^2(\Omega).
\end{align}
It remains to show that $(\kappa,u)\in\Xi$.
Since $\alpha\le \xi_1(x)\le\kappa_n(x)\le\xi_2(x)$ for a.e. $x\in\Omega$ and every $n\in\NN$, we have that $\kappa\in\mathfrak A_{ad}$. It also follows from Lemma \ref{lem-con} that
\begin{align}\label{Integra}
&\lim_{n\to\infty}\int_{\Omega}\int_{\Omega}\kappa_n(x-y)\frac{(u_n(x)-u_n(y))(\varphi(x)-\varphi(y))}{|x-y|^{N+2s}}\;dxdy\notag\\
=&\int_{\Omega}\int_{\Omega}\kappa(x-y)\frac{(u(x)-u(y))(\varphi(x)-\varphi(y))}{|x-y|^{N+2s}}\;dxdy,
\end{align}
for every $\varphi\in\mathcal D(\Omega)$. Let $\Phi:\Omega\times\Omega\to\RR$ be given by
$\displaystyle
\Phi(x,y):=\frac{|\varphi(x)-\varphi(y)|^{p-2}}{|x-y|^{s(p-2)}}$.
Note that for a.e. $x,y\in\Omega$, we have that
$|\Phi(x,y)|\le \|\varphi\|_{C^{0,s}(\bOm)}^{p-2}$.
Since $\Phi\in L^\infty(\Omega\times\Omega)$, if we multiply the functions under the integrals in both sides of \eqref{Integra} by $\Phi$, then we have the same convergence. 
This implies that for every $\varphi\in\mathcal D(\Omega)$,
\begin{align}\label{WM}
&\lim_{n\to\infty}\int_{\Omega}\int_{\Omega}\kappa_n(x-y)\frac{|\varphi(x)-\varphi(y)|^{p-2}}{|x-y|^{s(p-2)}}\frac{(\varphi(x)-\varphi(y))(u_n(x)-u_n(y))}{|x-y|^{N+2s}}\;dxdy\notag\\
=&\int_{\Omega}\int_{\Omega}\kappa(x-y)\frac{|\varphi(x)-\varphi(y)|^{p-2}}{|x-y|^{s(p-2)}}\frac{(\varphi(x)-\varphi(y))(u(x)-u(y))}{|x-y|^{N+2s}}\;dxdy\notag\\
=&\int_{\Omega}\int_{\Omega}\kappa(x-y)|\varphi(x)-\varphi(y)|^{p-2}\frac{(\varphi(x)-\varphi(y))(u(x)-u(y))}{|x-y|^{N+sp}}\;dxdy.
\end{align}

We show that  $(\kappa,u)$ is related by \eqref{eq-minty}. Since $(\kappa_n,u_n)$ satisfies \eqref{eq-minty} we have that

\begin{align}\label{eq-minty-2}
\mathcal E_{p,s}^{\kappa_n}(\varphi ,\varphi-u_n)\ge \int_{\Omega}f(\varphi-u_n)\;dx,
\end{align}
for every $\varphi\in\mathcal D(\Omega)$, where we recall that
$\displaystyle\mathcal E_{p,s}^{\kappa_n}(\varphi ,\varphi-u_n)=\mathcal E_{p,s}^{\kappa_n}(\varphi ,\varphi)-\mathcal E_{p,s}^{\kappa_n}(\varphi ,u_n)$.
It follows from \eqref{eq-33} that
\begin{align}\label{MW2}
\lim_{n\to\infty}\int_{\Omega}f(\varphi-u_n)\;dx= \int_{\Omega}f\varphi\;dx -\lim_{n\to\infty}\int_{\Omega}fu_n\;dx=\int_{\Omega}f(\varphi-u)\;dx.
\end{align}
Using \eqref{WM} and \eqref{MW2} we can pass to the limit in \eqref{eq-minty-2} as $n\to\infty$ and obtain that $(\kappa,u)$ is related by \eqref{eq-minty} for every $\varphi\in\mathcal D(\Omega)$. Finally, since $\mathcal D(\Omega)$ is dense in $W_0^{s,p}(\Omega)$, we have that \eqref{eq-minty} also holds for every $\varphi\in W_0^{s,p}(\Omega)$. Hence, $u\in W_0^{s,p}(\Omega)$ is a weak solution of  \eqref{ellip-eq}. This, together with $\kappa\in \mathfrak{A}_{ad}$ imply $(\kappa,u)\in\Xi$. 
\end{proof}

Now we are able to give the proof of our first main result.

\begin{proof}[\bf Proof of Theorem \ref{theo-24}]
Since the set $\Xi$ is nonempty and the cost functional is bounded from below on $\Xi$, it follows that there exists a minimizing sequence $(\kappa_n,u_n)\in\Xi$ to the problem \eqref{eq-min}, that is,
\begin{align*}
\inf_{(\kappa,u)\in\Xi}\mathbb I(\kappa,u)=\lim_{n\to\infty}\left[\frac 12\int_{\Omega}|u_n-\xi|^2\;dx+\int_{\Omega}|\nabla \kappa_n|\right]<\infty.
\end{align*}
This implies that $\{(\kappa_n,u_n)\}_{n\in \NN}$ is bounded in $BV(\Omega)\times W_0^{s,p}(\Omega)$. It follows from Theorem \ref{Theo-con} that after a (sub)sequence if necessary,  there exists $(\kappa_\star, u_\star)\in\Xi$ such that $\kappa_n\weak \kappa_\star$ in $BV(\Omega)$, $u_n\rightharpoonup u_\star$ in $W_0^{s,p}(\Omega)$ and $u_n\to u_\star$ in $L^2(\Omega)$, as $n\to\infty$. Therefore using also Remark \ref{rem-21}, we get that
\begin{align*}
\lim_{n\to\infty}\frac 12\int_{\Omega}|u_n-\xi|^2\;dx=\frac 12\int_{\Omega}|u_\star-\xi|^2\;dx\;\mbox{ and }\; \int_{\Omega}|\nabla \kappa_\star|\le\liminf_{n\to\infty}\int_{\Omega}|\nabla \kappa_n|.
\end{align*}
We have shown that 
$\displaystyle
\mathbb I(\kappa_\star,u_\star)\le \inf_{(\kappa,u)\in\Xi}\mathbb I(\kappa,u)$.
Thus $(\kappa_\star,u_\star)$ is a solution to \eqref{eq-min} and hence, a solution to our initial {\bf OCP} \eqref{Min}-\eqref{ellip-eq}. The proof is finished.
\end{proof}

\section{Proof of Theorem \ref{theo-ex-ep}}\label{proof-ex-ep}

Here also in order to be able to prove our theorem we need some preparation. For $\phi\in W_0^{s,2}(\Omega)$,  $p\in[2,\infty)$, $n\in\NN$ and $\varepsilon>0$ a small parameter, we shall use the following notation:

\begin{align}
\|\phi\|_{\varepsilon,n,\kappa,s,p}:=\left(\int_{\Omega}\int_{\Omega}\kappa(x-y)\Big[\varepsilon+\mathcal G_n(\phi,s)\Big]^{\frac{p-2}{2}}\frac{|\phi(x)-\phi(y)|^2}{|x-y|^{N+2s}}\;dxdy\right)^{\frac 1p}
\end{align}
where we recall that
$\mathcal G_n(\phi,s):=\mathcal F_n\left(\frac{|\phi(x)-\phi(y)|^2}{|x-y|^{2s}}\right).$
We notice that $\|\cdot\|_{\varepsilon,n,\kappa,s,p}$ is a quasi-norm but is not a norm unless $p=2$.

Let $\omega:\;\Omega\times\Omega\to \RR$ be the function and $\mu$ the measure on $\Omega\times\Omega$ given by
\begin{align}\label{mes-func}
\omega(x,y):=\frac{1}{|x-y|^{N+2s-2}}\;\;\mbox{ and }\; d\mu(x,y):=\omega(x,y)dxdy.
\end{align}
Let $x\in \Omega$ fixed and $R>0$ such that $\Omega\subset B(x,R)$. Using polar coordinates we get that there exists a constant $C>0$ (depending only on $N$ and $s$) such that
\begin{multline*}
\mu(\Omega\times\Omega):=\int_{\Omega}\int_{\Omega}\omega(x,y)dxdy=\int_{\Omega}\int_{\Omega}\frac{1}{|x-y|^{N+2s-2}}\;dxdy\\
\le \int_{\Omega}dx\int_{B(x,R)}\frac{1}{|x-y|^{N+2s-2}}\;dy
\le C|\Omega|\int_0^{R}\frac{1}{r^{2s-1}}\;dr=\frac{C|\Omega|}{2(1-s)}R^{2(1-s)}<\infty.
\end{multline*}

For $u\in W_0^{s,p}(\Omega)$ or $u\in W_0^{s,2}(\Omega)$ fixed and $n\in\NN$, we consider the level set
\begin{align*}
(\Omega\times\Omega)_n(u):=\left\{(x,y)\in\Omega\times\Omega:\; \frac{|u(x)-u(y)|}{|x-y|^{s}}>\sqrt{n^2+1}\right\}.
\end{align*}

We have the following result.

\begin{lemma}
The following assertions hold. 
\begin{enumerate}
\item There exists a constant $C>0$  such that for $u\in W_0^{s,2}(\Omega)$, 
\begin{align}\label{eq-48-2}
\Big|(\Omega\times\Omega)_n(u)\Big|\le \frac{C\alpha^{-1}}{n^p}\|u\|_{\varepsilon,n,\kappa,s,p}^{p}.
\end{align}

\item There exists a constant $C>0$  such that for $u\in W_0^{s,2}(\Omega)$, 
\begin{align}\label{eq-48}
\mu\Big((\Omega\times\Omega)_n(u)\Big)\le \frac{C\alpha^{-1}}{n^p}\|u\|_{\varepsilon,n,\kappa,s,p}^{p}.
\end{align}
\end{enumerate}
\end{lemma}

\begin{proof}
Let $u\in W_0^{s,2}(\Omega)$ and $p\in [2,\infty)$. 

(a) Using the H\"older inequality and \eqref{eq-F} we get that there exists a constant $C>0$ such that
\begin{align}\label{mq}
\Big|(\Omega\times\Omega)_n(u)\Big|=&\int\int_{(\Omega\times\Omega)_n(u)} 1\;dxdy
\le\frac{1}{\sqrt{n^2+1}}\int_{\Omega}\int_{\Omega} |x-y|^{\frac N2}\frac{|u(x)-u(y)|}{|x-y|^{\frac N2+s}}\;dxdy\notag\\
\le &\frac{C}{n}\Big|(\Omega\times\Omega)_n(u)\Big|^{\frac{1}{2}}\left(\int_{\Omega}\int_{\Omega} \frac{|u(x)-u(y)|^2}{|x-y|^{N+2s}}\;dxdy\right)^{\frac 12}\notag\\
\le &\frac{C}{n}\left(\frac{1}{\varepsilon+n^2+1}\right)^{\frac{p-2}{4}}\Big|(\Omega\times\Omega)_n(u)\Big|^{\frac{1}{2}}\alpha^{-\frac 12}\notag\\
&\times\left(\int_{\Omega}\int_{\Omega}\kappa(x-y)\Big[\varepsilon+\mathcal G_n\left(u,s\right)\Big]^{\frac{p-2}{2}}\frac{|u(x)-u(y)|^2}{|x-y|^{N+2s}}\;dxdy\right)^{\frac 12}\notag\\
\le &\frac{C}{n^{\frac p2}}\frac{1}{\sqrt{\alpha}}\Big|(\Omega\times\Omega)_n(u)\Big|^{\frac{1}{2}}\|u\|_{\varepsilon,n,\kappa,s,p}^{\frac p2}.
\end{align}
We have shown \eqref{eq-48-2}.

(b) Let $\omega$ be the weighted function and $\mu$ the measure given in \eqref{mes-func}. Proceeding as in \eqref{mq} we get that there exists a constant $C>0$ such that
\begin{align*}
\mu\Big((\Omega\times\Omega)_n(u)\Big)=&\frac{1}{\sqrt{n^2+1}}\int\int_{(\Omega\times\Omega)_n(u)}\frac{\sqrt{n^2+1}}{|x-y|^{N+2s-2}}\;dxdy\\
\le &\frac{C}{n}\mu\Big((\Omega\times\Omega)_n(u)\Big)^{\frac{1}{2}}\left(\int_{\Omega}\int_{\Omega} \frac{|u(x)-u(y)|^2}{|x-y|^{N+2s}}\;dxdy\right)^{\frac 12} \\
\le &\frac{C}{n}\left(\frac{1}{\varepsilon+n^2+1}\right)^{\frac{p-2}{4}}\mu\Big((\Omega\times\Omega)_n(u)\Big)^{\frac{1}{2}}\alpha^{-\frac 12}\\
&\times\left(\int_{\Omega}\int_{\Omega}\kappa(x-y)\Big[\varepsilon+\mathcal G_n\left(u,s\right)\Big]^{\frac{p-2}{2}}\frac{|u(x)-u(y)|^2}{|x-y|^{N+2s}}\;dxdy\right)^{\frac 12}\\
\le &\frac{C}{n^{\frac p2}}\frac{1}{\sqrt{\alpha}}\mu\Big((\Omega\times\Omega)_n(u)\Big)^{\frac{1}{2}}\|u\|_{\varepsilon,n,\kappa,s,p}^{\frac p2}.
\end{align*}
We have shown \eqref{eq-48} and the proof is finished.
\end{proof}

\subsection{Well-posedness of the regularized problem}\label{s:regstate}

Next, we show the existence of solution to the regularized state equation \eqref{ellip-ep}.

\begin{proposition}\label{exist-ep}
For every $\varepsilon>0$, $n\in\NN$, $\kappa\in\mathfrak{A}_{ad}$ and $f\in L^2(\Omega)$, the system \eqref{ellip-ep} has a unique weak solution $u_{\varepsilon,n}\in W_0^{s,2}(\Omega)$. In addition, there exists a constant $C>0$ {\em(depending only on $\Omega$ and $s$)} such that
\begin{align}\label{sol-epsilon}
\varepsilon^{\frac{p-2}{2}}\|u_{\varepsilon,n}\|_{W_0^{s,2}(\Omega)}\le C\|f\|_{L^2(\Omega)}.
\end{align}
\end{proposition}

\begin{proof}
Here also, the proposition follows by showing that $\mathbb F_{\varepsilon,n,p}^\kappa(u,\cdot)\in W^{-s,2}(\Omega)$ for every fixed $u\in W_0^{s,2}(\Omega)$, and that $\mathbb F_{\varepsilon,n,p}^\kappa$ is hemi-continuous, strictly monotone and coercive. The estimates \eqref{sol-epsilon} follows by taking $v=u_{\varepsilon,n}$ as a test function in \eqref{solu-epsi}. For more details we refer to \cite[Proposition 2.7]{AW2017note}. The proof is finished.
\end{proof}

\begin{remark}[\bf The regularized state equation and Minty relation]\label{rem-45}
{\em
As in Remark \ref{rem} we have that $u_{\varepsilon,n}\in W_0^{s,2}(\Omega)$ satisfies \eqref{solu-epsi} if and only if the following {\bf Minty relation} holds, that is, for every $\varphi\in W_0^{s,2}(\Omega)$,
\begin{align}\label{eq-minty-ep}
\mathbb F_{\varepsilon,n,p}^\kappa(\varphi ,\varphi-u_{\varepsilon,n})\ge \int_{\Omega}f(\varphi-u_{\varepsilon,n})\;dx.
\end{align}
For more details see \cite[Remark 2.9]{AW2017note}.}
\end{remark}

\subsection{The regularized optimal control problem (ROCP)}\label{s:regOCP_exist}

We begin by introducing the set of admissible controls for the {\bf ROCP} \eqref{eq-27}-\eqref{eq-29}. That is,
\begin{align}\label{set-e-ep}
\Xi_{\varepsilon,n}=\Big\{(\kappa,u):\; \kappa\in\mathfrak{A}_{ad},\; u\in W_0^{s,2}(\Omega),\; (\kappa,u)\mbox{ are related by }\;\eqref{solu-epsi}\Big\} . 
\end{align}
It follows from Remark \ref{rem-45} that the set $\Xi_{\varepsilon,n}$ is nonempty for every $\varepsilon>0$ and $n\in\NN$. Therefore the {\bf ROCP} \eqref{eq-27}-\eqref{eq-29} can be rewritten as the minimization problem: 
\begin{align}\label{cont-ep-pro}
\min_{(\kappa,u)\in \Xi_{\varepsilon,n}}\mathbb I(\kappa,u).
\end{align}

Now we are ready to  give the proof of our second main result.

\begin{proof}[\bf Proof of Theorem \ref{theo-ex-ep}]
Since the set $\Xi_{\varepsilon,n}$ given in \eqref{set-e-ep} is nonempty, then we can take a minimizing sequence $\{(\kappa_k,u_k)\}_{k\in\NN}\subset \Xi_{\varepsilon,n}$.  As $\{\kappa_k\}_{k\in\NN}$ is bounded in $L^\infty(\Omega)$ and $0<\alpha\le \kappa_k(x)\le\xi_2(x)$ for a.e. $x\in\Omega$ and every $k\in\NN$, we have that there exists a constant $C>0$ (independent of $k$)  such that
\begin{align*}
\|\kappa_k\|_{BV(\Omega)}\le \|\xi_2\|_{L^1(\Omega)}+C.
\end{align*}
Using the lower boundedness of $\mathbb I$, the above estimate and \eqref{sol-epsilon}, we have that there exists a constant $C>0$ (independent of $k$) such that
\begin{align*}
\|\kappa_k\|_{BV(\Omega)}+\|u_k\|_{W_0^{s,2}(\Omega)}\le C\left(\|\xi_2\|_{L^1(\Omega)}+1+\varepsilon^{\frac{2-p}{2}}\|f\|_{L^2(\Omega)}\right).
\end{align*}
Thus $\{(\kappa_k,u_k)\}_{k\in\NN}$ is bounded in $BV(\Omega)\times W_0^{s,2}(\Omega)$. Therefore, proceeding as the proof of Theorem \ref{theo-24} we deduce the existence of a (sub)sequence still denoted by $\{(\kappa_k,u_k)\}_{k\in\NN}$, and a pair $(\kappa,u)\in \Xi_{\varepsilon,n}$ such that $\kappa_k\weak \kappa$ in $BV(\Omega)$, $u_k\rightharpoonup u$ in $W_0^{s,2}(\Omega)$ and $u_k\to u$ in $L^2(\Omega)$, as $k\to\infty$. Thus
$\mathbb I(\kappa,u)\le \liminf_{k\to\infty}\mathbb I(\kappa_k,u_k)$.
The proof of the theorem is finished.
\end{proof}

\subsection{Further a priori estimates for the regularized state}\label{s:reg_state_apriori}

Next we give further a priori estimates of weak solutions to the system \eqref{ellip-ep}. These results will be useful to show the convergence of solutions to the {\bf ROCP} in section~\ref{proof-23}.

\begin{proposition}\label{prop-46}
Let $\kappa\in\mathfrak{A}_{ad}$, $n\in\NN$ and $\varepsilon>0$ be given. Then there exists a constant $C>0$ such that for every $f\in L^2(\Omega)$ and $u\in W_0^{s,2}(\Omega)$,
\begin{align}\label{AB}
\left|\int_{\Omega}fu\;dx\right|\le C\|f\|_{L^2(\Omega)}\left[\alpha^{-\frac 1p}\|u\|_{\varepsilon,n,\kappa,s,p}+\alpha^{-\frac 12}\|u\|_{\varepsilon,n,\kappa,s,p}^{\frac p2}\right].
\end{align}
\end{proposition}

\begin{proof}
We associate with $u\in W_0^{s,2}(\Omega)$ the level set $(\Omega\times\Omega)^n(u)$ given by
\begin{align}\label{omn}
(\Omega\times\Omega)^n(u):=\left\{(x,y)\in\Omega\times\Omega:\; \frac{|u(x)-u(y)|}{|x-y|^s}>n\right\}.
\end{align}
Let $\frac 12<t\le  s<1$ and $p\in [2,\infty)$ with $t=s$ if $p=2$.
Then $ W_0^{s,2}(\Omega) \hookrightarrow  W_0^{t,2}(\Omega)$ (by Proposition~\ref{rem-sob-hol}) and hence, $u\in W_0^{t,2}(\Omega)$. Using \eqref{sob-emb} and \eqref{norm-26}, we get that there exists a constant $C>0$ such that
\begin{align}\label{AB1}
&\left|\int_{\Omega}fu\;dx\right|\le\|f\|_{L^2(\Omega)}\|u\|_{L^2(\Omega)}
\le C\|f\|_{L^2(\Omega)}\|u\|_{W_0^{t,2}(\Omega)}\notag\\
\le &C\|f\|_{L^2(\Omega)}\left[\left(\int\int_{(\Omega\times\Omega)\setminus(\Omega\times\Omega)^n(u) }\frac{|u(x)-u(y)|^2}{|x-y|^{N+2t}}\;dxdy\right)^{\frac 12}\right.\notag\\
&+\left.\left(\int\int_{(\Omega\times\Omega)^n(u)}\frac{|u(x)-u(y)|^2}{|x-y|^{N+2t}}\;dxdy\right)^{\frac 12}\right].
\end{align}
Using the H\"older inequality we get that there is a constant $C>0$ such that
\begin{align}\label{AB2}
&\left(\int\int_{(\Omega\times\Omega)\setminus(\Omega\times\Omega)^n(u) }\frac{|u(x)-u(y)|^2}{|x-y|^{N+2t}}\;dxdy\right)^{\frac 12}\notag\\
\le &\left(\int\int_{(\Omega\times\Omega)\setminus(\Omega\times\Omega)^n(u) }\left(\frac{1}{|x-y|^{N-\frac{2N}{p}+2(t-s)}}\right)^{\frac{p}{p-2}}\;dxdy\right)^{\frac{p-2}{2p}}\notag\\
&\times\left(\int\int_{(\Omega\times\Omega)\setminus(\Omega\times\Omega)^n(u)}\frac{|u(x)-u(y)|^p}{|x-y|^{N+sp}}\;dxdy\right)^{\frac 1p}\notag\\
\le &C\left(\int\int_{(\Omega\times\Omega)\setminus(\Omega\times\Omega)^n(u)}\frac{|u(x)-u(y)|^p}{|x-y|^{N+sp}}\;dxdy\right)^{\frac 1p}.
\end{align}
It follows from \eqref{AB2} that there exists a constant $C>0$ such that
\begin{align}\label{AB4}
&\left(\int\int_{(\Omega\times\Omega)\setminus(\Omega\times\Omega)^n(u) }\frac{|u(x)-u(y)|^2}{|x-y|^{N+2t}}\;dxdy\right)^{\frac 12}\\
&\le C\left(\int\int_{(\Omega\times\Omega)\setminus(\Omega\times\Omega)^n(u)}\left(\varepsilon +\frac{|u(x)-u(y)|^2}{|x-y|^{2s}}\right)^{\frac{p-2}{2}}\frac{|u(x)-u(y)|^2}{|x-y|^{N+2s}}\;dxdy\right)^{\frac 1p}.\notag
\end{align}
Since
$\displaystyle
\mathcal F_n\left(\frac{|u(x)-u(y)|^2}{|x-y|^{2s}}\right)=\frac{|u(x)-u(y)|^2}{|x-y|^{2s}}\;\;\mbox{ a.e. in }\; (\Omega\times\Omega)\setminus(\Omega\times\Omega)^n(u)$,
and $0<\alpha\le \kappa(x-y)$ for a.e. $x,y\in\Omega$, it follows from \eqref{AB4} that
\begin{align}\label{AB5}
&\left(\int\int_{(\Omega\times\Omega)\setminus(\Omega\times\Omega)^n(u) }\frac{|u(x)-u(y)|^2}{|x-y|^{N+2t}}\;dxdy\right)^{\frac 12}\notag\\
&\le C\alpha^{-\frac 1p}\left(\int\int_{(\Omega\times\Omega)\setminus(\Omega\times\Omega)^n(u)}\kappa(x-y)\Big[\varepsilon +\mathcal G_n\left(u,s\right)\Big]^{\frac{p-2}{2}}\frac{|u(x)-u(y)|^2}{|x-y|^{N+2s}}\;dxdy\right)^{\frac 1p}\notag\\
&\le C\alpha^{-\frac 1p}\|u\|_{\varepsilon,n,\kappa,s,p}.
\end{align}
Since
$\displaystyle
n^2\le \mathcal F_n\left(\frac{|u(x)-u(y)|^2}{|x-y|^{2s}}\right)\le n^2+1\;\;\mbox{ a.e. in }\; (\Omega\times\Omega)^n(u)$,
and $\frac 12<t\le s<1$, we have that there exists a constant $C>0$ such that
\begin{align}\label{AB6}
&\left(\int\int_{(\Omega\times\Omega)^n(u)}\frac{|u(x)-u(y)|^2}{|x-y|^{N+2t}}\;dxdy\right)^{\frac 12}\notag\\
\le &C\left(\int\int_{(\Omega\times\Omega)^n(u)}\frac{|u(x)-u(y)|^2}{|x-y|^{N+2s}}\;dxdy\right)^{\frac 12}\notag\\
\le &C\alpha^{-\frac 12}\left(\int\int_{(\Omega\times\Omega)^n(u)}\kappa(x-y)\Big[\varepsilon +\mathcal G_n\left(u,s\right)\Big]^{\frac{p-2}{2}}\frac{|u(x)-u(y)|^2}{|x-y|^{N+2s}}\;dxdy\right)^{\frac 12}\notag\\
\le& C\alpha^{-\frac 12}\|u\|_{\varepsilon,n,\kappa,s,p}^{\frac p2}.
\end{align}
Now \eqref{AB} follows from \eqref{AB1}, \eqref{AB5} and \eqref{AB6}. The proof is finished.
\end{proof}

\begin{proposition}
Let $\varepsilon>0$ and $n\in\NN$. Then for every $\kappa\in\mathfrak{A}_{ad}$ and $f\in L^2(\Omega)$, the sequence of weak solution $\{u_{\varepsilon,n}\}_{{\varepsilon>0, n\in\NN}}$ to the system \eqref{ellip-ep} is bounded with respect to the $\|\cdot\|_{\varepsilon,n,\kappa,s,p}$-quasi norm, that is,
\begin{align}\label{quasi-}
\sup_{\varepsilon>0,n\in\NN}\|u_{\varepsilon,n}\|_{\varepsilon,n,\kappa,s,p}<\infty.
\end{align}
\end{proposition}

\begin{proof}
Using \eqref{solu-epsi} and \eqref{AB} we get that there is a constant $C>0$ such that
\begin{align}\label{B1}
&\|u_{\varepsilon,n}\|_{\varepsilon,n,\kappa,s,p}^p\le\mathbb F_{\varepsilon,n,p}^{\kappa}(u_{\varepsilon,n},u_{\varepsilon,n})=\int_{\Omega}fu_{\varepsilon,n}\;dx\notag\\
\le &C\|f\|_{L^2(\Omega)}\Big[\alpha^{-\frac 1p}\|u\|_{\varepsilon,n,\kappa,s,p}+\alpha^{-\frac 12}\|u\|_{\varepsilon,n,\kappa,s,p}^{\frac p2}\Big].
\end{align}
Let
$\displaystyle
C_f:=C\Big(\alpha^{-\frac 1p}+\alpha^{-\frac 12}\Big)\|f\|_{L^2(\Omega)}$.
It follows from \eqref{B1} that
\begin{align}\label{B2}
\|u_{\varepsilon,n}\|_{\varepsilon,n,\kappa,s,p}\le \max\left\{C_f^{\frac 2p},C_f^{\frac{1}{p-1}}\right\},\;\forall\;\varepsilon>0, \forall\;n\in\NN,\;\forall\; \kappa\in\mathfrak{A}_{ad}.
\end{align}
Now \eqref{quasi-} follows from \eqref{B2} and the proof is finished.
\end{proof}

We also mention that using \eqref{AB} and \eqref{B2} we get that
\begin{align}
\|u_{\varepsilon,n}\|_{L^2(\Omega)}\le \max\left\{C^2\|f\|_{L^2(\Omega)}, C^{\frac{p}{p-1}}\|f\|_{L^2(\Omega)}^{\frac{1}{p-1}}\right\},\;\forall\;\varepsilon>0, \;\forall\;n\in\NN,
\end{align}
where $C$ is the constant appearing in \eqref{AB}.

We conclude this section with the following remark.

\begin{remark}\label{rem-49}
{\em Let $\kappa\in\mathfrak{A}_{ad}$, $n\in\NN$, $\varepsilon>0$ and $u_{\varepsilon,n}\in W_0^{s,2}(\Omega)$ be the solution of \eqref{ellip-ep}. Let $(\Omega\times\Omega)^n(u_{\varepsilon,n})$ be given by \eqref{omn}. Let $p\in [2,\infty)$ and $\frac 12<t\le s<1$ with $t=s$ if $p=2$.
We notice that it follows from \eqref{AB5} and \eqref{AB6} that there exists a constant $C>0$ (depending on $\Om, N,s,t$ and $p$) such that 
\begin{align}\label{D1}
\|u_{\varepsilon,n}\|_{W_0^{t,2}(\Omega)}\le C\Big(\alpha^{-\frac 1p}\|u_{\varepsilon,n}\|_{\varepsilon,n,\kappa,s,p}+\alpha^{-\frac 12}\|u_{\varepsilon,n}\|_{\varepsilon,n,\kappa,s,p}^{\frac p2}\Big).
\end{align}
We do not know if \eqref{D1} holds with $W_0^{t,2}(\Omega)$ replaced by $W_0^{s,2}(\Omega)$ in the case $p>2$.
}
\end{remark}

\section{Proof of Theorem \ref{theorem-23}}\label{proof-23}

Before we give the proof of our last main result, i.e., Theorem \ref{theorem-23} we need some intermediate results.

\begin{lemma}\label{lem-51}
Let $\frac 12<t\le s<1$ and $p\in [2,\infty)$ with $t=s$ if $p=2$.
Let $\{\kappa_{\varepsilon,n}\}_{\varepsilon>0, n\in\NN}\subset\mathfrak{A}_{ad}$ be an arbitrary sequence of admissible control associated with the states $\{u_{\varepsilon,n}\}_{\varepsilon>0, n\in\NN}\subset W_0^{s,2}(\Omega)$. Then $\{u_{\varepsilon,n}\}_{\varepsilon>0, n\in\NN}$ is bounded in $W_0^{t,2}(\Omega)$. In addition, each cluster point $u$ of $\{u_{\varepsilon,n}\}_{\varepsilon>0, n\in\NN}$ with respect to the weak topology in $W_0^{t,2}(\Omega)$, belongs to $W_0^{s,p}(\Omega)$. 
\end{lemma}

\begin{proof}
 Recall that the estimate \eqref{D1} in Remark \ref{rem-49} holds.
Now using \eqref{B2} we get from \eqref{D1} that $\{u_{\varepsilon,n}\}_{\varepsilon>0, n\in\NN}$ is bounded in $W_0^{t,2}(\Omega)$. Let  $\{u_{\varepsilon_k,n_k}\}_{k\in\NN}$ be a (sub)sequence and $u\in W_0^{t,2}(\Omega)$ be such that $u_{\varepsilon_k,n_k}\rightharpoonup u$ in $W_0^{t,2}(\Omega)$ and $u_{\varepsilon_k,n_k}\to u$ in $L^2(\Omega)$, as $k\to\infty$. Let $k\in\NN$ be fixed and set
\begin{align}
\mathbb B_k:=\bigcup_{j=k}^\infty(\Omega\times\Omega)_{n_j}(u_{\varepsilon_j,n_j}),
\end{align}
where we recall that 
\begin{align}\label{D2}
(\Omega\times\Omega)_{n_j}(u_{\varepsilon_j,n_j}):=\Big\{(x,y)\in\Omega\times\Omega:\; \frac{|u_{\varepsilon_j,n_j}(x)-u_{\varepsilon_j,n_j}(y)|}{|x-y|^{s}}>\sqrt{n_j^2+1}\Big\}.
\end{align}
Using \eqref{eq-48} and \eqref{B2} we get that
\begin{align*}
\Big|\mathbb B_k\Big|=&\alpha^{-1}\sum_{j=k}^\infty\frac{1}{n_j^p}\|u_{\varepsilon_j,n_j}\|_{\varepsilon_j,n_j,\kappa_{\varepsilon_j},s}^p\le \alpha^{-1}\sup_{j\in\NN}\|u_{\varepsilon_j,n_j}\|_{\varepsilon_j,n_j,\kappa_{\varepsilon_j},s}^p\sum_{j=k}^\infty\frac{1}{n_j^p}\\
\le &\alpha^{-1} \max\left\{C_f^{\frac 2p},C_f^{\frac{1}{p-1}}\right\}\sum_{j=k}^\infty\frac{1}{n_j^p}<\infty.
\end{align*}
Therefore,
\begin{align}\label{mea-zero}
\lim_{k\to\infty}\Big|\mathbb B_k\Big|=\Big|\limsup_{k\to\infty}\mathbb B_k\Big|=0.
\end{align}
Using \eqref{B2} again we get that for all $j\ge k$,
\begin{align}\label{D3}
&\int\int_{(\Omega\times\Omega)\setminus\mathbb B_k}\frac{|u_{\varepsilon_j,n_j}(x)-u_{\varepsilon_j,n_j}(y)|^p}{|x-y|^{N+sp}}\;dxdy\notag\\
\le &\frac{1}{\alpha}\int\int_{(\Omega\times\Omega)\setminus\mathbb B_k}\kappa_{\varepsilon_j,n_j}(x-y)\left[\varepsilon_j+\mathcal G_{n_j}\left(u_{\varepsilon_j,n_j},s\right)\right]^{\frac{p-2}{2}} \frac{|u_{\varepsilon_j,n_j}(x)-u_{\varepsilon_j,n_j}(y)|^2}{|x-y|^{N+2s}}\;dxdy\notag\\
\le &\alpha^{-1}\max\left\{C_f^{\frac 2p},C_f^{\frac{1}{p-1}}\right\}.
\end{align}
Let $U_{\varepsilon_j,n_j,(p,s)}, U_{(p,s)}:\Omega\times\Omega\to\RR$ be given as in \eqref{Up}.
It follows from \eqref{D3} that $\displaystyle\left\{U_{\varepsilon_j,n_j,(p,s)}\right\}_{j\in\NN}$ is bounded in $L^p((\Omega\times\Omega)\setminus\mathbb B_k)$. 
Since $u_{\varepsilon_j,n_j}\rightharpoonup u$ in $W_0^{t,2}(\Omega)$ as $j\to \infty$, then by Remark \ref{rem-1},
\begin{align}\label{con-t}
U_{\varepsilon_j,n_j,(2,t)}\rightharpoonup U_{(2,t)}\;\mbox{ in }\;L^2(\Omega\times\Omega)\;\mbox{ as }\;j\to\infty.
\end{align}
Let
$\displaystyle
\beta:=\frac N2-\frac Np+t-s$.
Since
$\displaystyle
U_{\varepsilon_j,n_j,(p,s)}(x,y)= |x-y|^{\beta} U_{\varepsilon_j,n_j,(2,t)}(x,y)$
and
$\displaystyle
U_{(p,s)}(x,y)=|x-y|^{\beta}U_{(2,t)}(x,y)$,
it follows from \eqref{con-t} that
\begin{align*}
\chi_{(\Omega\times\Omega)\setminus \mathbb B_k}U_{\varepsilon_j,n_j,(p,s)}\rightharpoonup U_{(p,s)}\;\mbox{ in }\;L^p(\Omega\times\Omega)\;\mbox{ as }\;j\to\infty.
\end{align*}
Using \eqref{mea-zero} and \eqref{D3} we get that
\begin{align}\label{p-est}
\int_{\Omega}\int_{\Omega}\frac{|u(x)-u(y)|^p}{|x-y|^{N+sp}}\;dxdy=&\lim_{k\to\infty}\int\int_{(\Omega\times\Omega)\setminus\mathbb B_k}\frac{|u(x)-u(y)|^p}{|x-y|^{N+sp}}\;dxdy\notag\\
\le &\lim_{k\to\infty}\liminf_{j\to\infty} \int\int_{(\Omega\times\Omega)\setminus\mathbb B_k} \frac{|u_{\varepsilon_j,n_j}(x)-u_{\varepsilon_j,n_j}(y)|^p}{|x-y|^{N+sp}}\;dxdy\notag\\
\le &\alpha^{-1}\max\left\{C_f^{\frac 2p},C_f^{\frac{1}{p-1}}\right\}<\infty.
\end{align}
Since $u\in W_0^{t,2}(\Omega)$ and by assumption $\Omega$ has a Lipsichitz continuous boundary, then proceeding as in \cite[Corollary 1.5.2 p.37]{MP}, we can conclude from \eqref{p-est} that $u\in W_0^{s,p}(\Omega)$. The proof of the lemma is finished.
\end{proof}

\begin{remark}
{\em We mention that we do not know if $\{u_{\varepsilon,n}\}_{\varepsilon>0, n\in\NN}$ given in Lemma \ref{lem-51} is bounded in $W_0^{s,2}(\Omega)$ if $p>2$. We just know that it is bounded in $W_0^{t,2}(\Omega)$ for $\frac 12<t< s<1$.  In fact,  by \eqref{sol-epsilon}, we have that for $\varepsilon>0$ fixed, $\{u_{\varepsilon,n}\}_{n\in\NN}$ is bounded in $W_0^{s,2}(\Omega)$. But for $n\in\NN$ fixed, we are not able to show that $\{u_{\varepsilon,n}\}_{\varepsilon>0}$ is bounded in $W_0^{s,2}(\Omega)$ if $p>2$.}
\end{remark}

\begin{lemma}\label{lem-52}
Let $\frac 12<t\le s<1$and $p\in [2,\infty)$ with $t=s$ if $p=2$.
Let $\{\varepsilon_k\}_{k\in\NN}$, $\{n_k\}_{k\in\NN}$ and $\{\kappa_k\}_{k\in\NN}\subset\mathfrak{A}_{ad}$ be sequences such that
\begin{align}
\varepsilon_k\to 0,\;\;n_k\to\infty\;\mbox{ and }\; \kappa_k\to\kappa\;\mbox{  in }\; L^1(\Omega)\;\mbox{ as }\;k\to\infty.
\end{align}
Let $u_{k}=u_{\varepsilon_k,n_k}(\kappa_k)$ and $u=u(\kappa)$ be the solutions of \eqref{ellip-ep} and \eqref{ellip-eq}, respectively. Let $(\Omega\times\Omega)_k(u_k)$ be given in \eqref{D2}. Let $U_{k,(p,s)}, U_{(p,s)}$  be given as in \eqref{Up}.
Then the following assertions hold.
\begin{align}
&u_k\to u\;\mbox{ in }\; \;W_0^{t,2}(\Omega)\;\mbox{ as }\; k\to\infty.\label{57}\\
&\chi_{(\Omega\times\Omega)\setminus (\Omega\times\Omega)_k(u_k)}U_{k,(p,s)}\to U_{(p,s)}\;\mbox{  in }\; L^p(\Omega\times\Omega) \;\mbox{ as }\; k\to\infty.\label{58}\\
&\lim_{k\to\infty}\int_{\Omega}\int_{\Omega}\kappa_k(x-y)\Big[\varepsilon_k +\mathcal G_{n_k}\left(u_{k},s\right)\Big]^{\frac{p-2}{2}}\frac{|u_{k}(x)-u_{k}(y)|^2}{|x-y|^{N+2s}}\;dxdy\notag\\
&=\int_{\Omega}\int_{\Omega}\kappa(x-y)\frac{|u(x)-u(y)|^p}{|x-y|^{N+sp}}\;dxdy,\label{59}
\end{align}
where we recall that $\mathcal G_{n_k}$ is given by \eqref{FUNC-G}.
\end{lemma}

\begin{proof}
We prove the lemma in several steps.

{\bf Step 1}. It follows from Lemma \ref{lem-51} that after a (sub)sequence if necessary, there exists $u\in W_0^{t,2}(\Omega)$ such that $u_k\rightharpoonup u$ in $W_0^{t,2}(\Omega)$ and $u_k\to u$ in $L^2(\Omega)$, as $k\to\infty$. In addition we have that $u\in W_0^{s,p}(\Omega)$. We show that $u$ is a weak solution of \eqref{ellip-eq}. Let $\varphi\in\mathcal D(\Omega)$ be fixed. Recall that $u_k\in W_0^{s,2}(\Omega)$ and satisfies the Minty relation
\begin{align}\label{minty-ep-2}
\mathbb F_{\varepsilon_k,n_k,p}^{\kappa_k}(\varphi ,\varphi-u_{k}):=\mathbb F_{k,p}^{\kappa_k}(\varphi ,\varphi-u_{k})\ge \int_{\Omega}f(\varphi-u_{k})\;dx.
\end{align}
Define the functions $\mathbb G_k, \mathbb G:\Omega\times\Omega\to\RR$ by
\begin{align*}
\mathbb G_k(x,y):=\left[\varepsilon_k+\mathcal F_k\left(\frac{|\varphi(x)-\varphi(y)|^2}{|x-y|^{2s}}\right)\right]^{\frac{p-2}{2}}\frac{\varphi(x)-\varphi(y)}{|x-y|^{\frac{N}{p'}+s}}
\end{align*}
and
\begin{align*}
\mathbb G(x,y):=\left(\frac{|\varphi(x)-\varphi(y)|}{|x-y|^s}\right)^{p-2}\frac{\varphi(x)-\varphi(y)}{|x-y|^{\frac{N}{p'}+s}}.
\end{align*}
Since $0\le \mathcal F_k(\tau)\le \tau +1$ for all $\tau\ge 0$ and $k\in\NN$, and $p'(p-1)=p$, we have that
\begin{align*}
&\|\mathbb G_k\|_{L^{p'}(\Omega\times\Omega)}^{p'}:=\int_{\Omega}\int_{\Omega}|\mathbb G_k|^{p'}\;dxdy\\
\le &\int_{\Omega}\int_{\Omega}\left(\varepsilon_k+1+\frac{|\varphi(x)-\varphi(y)|^2}{|x-y|^{2s}}\right)^{p'\frac{p-2}{2}}\frac{|\varphi(x)-\varphi(y)|^{p'}}{|x-y|^{N+sp'}}\;dxdy\\
\le &C\left(\int_{\Omega}\int_{\Omega}\left(\varepsilon_k+1\right)^{p'\frac{p-2}{2}}\frac{|\varphi(x)-\varphi(y)|^{p'}}{|x-y|^{N+sp'}}\;dxdy
+\int_{\Omega}\int_{\Omega}\frac{|\varphi(x)-\varphi(y)|^{p'(p-1)}}{|x-y|^{N+sp'(p-1)}}\;dxdy\right)\\
\le &C\left(\|\varphi\|_{W_0^{s,p'}(\Omega)}^{p'}+\|\varphi\|_{W_0^{s,p}(\Omega)}^p\right)<\infty,
\end{align*}
where we have also used the fact that there exists a constant $C>0$ such that
\begin{align*}
\left(\varepsilon_k+1+\frac{|\varphi(x)-\varphi(y)|^2}{|x-y|^{2s}}\right)^{p'\frac{p-2}{2}}\le C\left[ \left(\varepsilon_k+1\right)^{p'\frac{p-2}{2}} +\frac{|\varphi(x)-\varphi(y)|^{p'(p-2)}}{|x-y|^{sp'(p-2)}}\right],
\end{align*}
which follows from the well-known inequalities
\begin{equation}\label{JW2}
\begin{cases}
(a+b)^q\le 2^{q-1}(a^q+b^q),\;\;&\forall\; a, b\ge 0,\; q>1\\
 (a+b)^q\le a^q+b^q,\;\;&\forall\; a, b\ge 0,\; q\in (0,1].
 \end{cases}
\end{equation}
Thus $\mathbb G_k\in L^{p'}(\Omega\times\Omega)$. Similarly, we get that
$\mathbb G\in L^{p'}(\Omega\times\Omega)$. Using the Lebesgue Dominated Convergence Theorem, we get that
\begin{align}\label{S1}
\mathbb G_k\to \mathbb G\;\mbox{  in }\; L^{p'}(\Omega\times\Omega)\;\mbox{ as }\; k\to\infty.
\end{align}
Let $\mathbb K_{k,p}^\varphi, \mathbb K_p^\varphi:\Omega\times\Omega\to\RR$ be defined by
\begin{align*}
\mathbb K_{k,p}^\varphi(x,y):=\kappa_k(x-y)\frac{\varphi(x)-\varphi(y)}{|x-y|^{\frac{N}{p}+s}}\;\mbox{ and }\;\mathbb K_p^\varphi(x,y):= \kappa(x-y)\frac{\varphi(x)-\varphi(y)}{|x-y|^{\frac{N}{p}+s}}.
\end{align*}
Then
\begin{align*}
\int_{\Omega}\int_{\Omega}|\mathbb K_{k,p}^\varphi(x,y)|^p\;dxdy\le\|\kappa_k\|_{L^\infty(\Omega)}^p\|\varphi\|_{W_0^{s,p}(\Omega)}^p.
\end{align*}
Proceeding similarly, we get that $\mathbb K_{k,p}^\varphi, \mathbb K_{p}^\varphi\in L^p(\Omega\times\Omega)$.
Since $\{\kappa_k\}_{k\in\NN}$ is bounded in $L^\infty(\Omega)$, then using \eqref{EQ} and the Lebesgue Dominated Convergence Theorem, we get that
\begin{align}\label{S2}
\mathbb K_{k,p}^\varphi\to \mathbb K_p^\varphi\;\mbox{  in }\; L^p(\Omega\times\Omega) \;\mbox{ as }\; k\to\infty.
\end{align}
Using \eqref{S1} and \eqref{S2} we get that
\begin{align}\label{SS1}
&\lim_{k\to\infty}\int_{\Omega}\int_{\Omega}\kappa_k(x-y)\mathcal G_{\varepsilon_k,k,p}(\varphi)
\frac{(\varphi(x)-\varphi(y))(\varphi(x)-\varphi(y))}{|x-y|^{N+2s}}\;dxdy\notag\\
=&\lim_{k\to\infty}\int_{\Omega}\int_{\Omega}\mathbb G_k(x,y)\mathbb K_{k,p}^\varphi(x,y)\;dxdy
=\int_{\Omega}\int_{\Omega}\mathbb G(x,y)\mathbb K_p^\varphi(x,y)\;dxdy\notag\\
=&\int_{\Omega}\int_{\Omega}\kappa(x-y)\left(\frac{|\varphi(x)-\varphi(y)|}{|x-y|^s}\right)^{p-2}\frac{(\varphi(x)-\varphi(y))(\varphi(x)-\varphi(y))}{|x-y|^{N+2s}}\;dxdy.
\end{align}
Proceeding similarly and using $u_k\rightharpoonup u$ in $W_0^{t,2}(\Omega)$ as $k\to\infty$, we get that 
\begin{align}\label{SS2}
&\lim_{k\to\infty}\int_{\Omega}\int_{\Omega}\kappa_k(x-y)\mathcal G_{\varepsilon_k,k,p}(\varphi)
\frac{(\varphi(x)-\varphi(y))(u_k(x)-u_k(y))}{|x-y|^{N+2s}}\;dxdy\notag\\
=&\int_{\Omega}\int_{\Omega}\kappa(x-y)\left(\frac{|\varphi(x)-\varphi(y)|}{|x-y|^s}\right)^{p-2}\frac{(\varphi(x)-\varphi(y))(u(x)-u(y))}{|x-y|^{N+2s}}\;dxdy.
\end{align}
Since $u_k\to u$ in $L^2(\Om)$ as $k\to\infty$,  we have that 
\begin{align}\label{SS3}
\lim_{k\to\infty}\int_{\Omega}\varphi(\varphi-u_k)\;dx=\int_{\Omega}\varphi(\varphi-u)\;dx.
\end{align}
Combining \eqref{SS1}, \eqref{SS2} and \eqref{SS3}, we can pass to the limit in \eqref{minty-ep-2} to get that $u$ satisfies the Minty relation \eqref{eq-minty} for every $\varphi\in\mathcal D(\Omega)$. Since $\mathcal D(\Omega)$ is dense in $W_0^{s,p}(\Omega)$, we have that \eqref{eq-minty} also holds for every $\varphi\in W_0^{s,p}(\Omega)$. We have shown that $u$ is a weak solution to the system \eqref{ellip-eq}. From the uniqueness of solution to  \eqref{ellip-eq}, we can deduce that the whole sequence $\{u_k\}$ converges weakly to $u=u(\kappa)$ in $W_0^{t,2}(\Omega)$, and hence converges strongly to $u=u(\kappa)$ in $L^2(\Omega)$, as $k\to\infty$.

{\bf Step 2}. We show \eqref{58}. Using \eqref{B2} we get that for every $k\in\NN$,
\begin{align}\label{ESTS}
&\int_{\Omega}\int_{\Omega}\chi_{(\Omega\times\Omega)\setminus (\Omega\times\Omega)_k(u_k)}\frac{|u_{k}(x)-u_{k}(y)|^p}{|x-y|^{N+sp}}\;dxdy\notag\\
\le &\alpha^{-1}\int_{\Omega}\int_{\Omega}\chi_{(\Omega\times\Omega)\setminus (\Omega\times\Omega)_k(u_k)}\kappa_{k}(x-y)\Big[\varepsilon_k+\mathcal G_k(u_k,s)\Big]^{\frac{p-2}{2}} \frac{|u_k(x)-u_k(y)|^2}{|x-y|^{N+2s}}\;dxdy\notag\\
\le &\alpha^{-1}\|u_k\|_{\varepsilon_k,n_k,\kappa_k,s,p}\le C<\infty.
\end{align}
It follows from \eqref{ESTS} that $\{\chi_{(\Omega\times\Omega)\setminus (\Omega\times\Omega)_k(u_k)}U_{k,(p,s)}\}_{k\in\NN}$ is  bounded in $L^p(\Omega\times\Omega)$. Therefore, after a (sub)sequence if necessary, there exists a $G\in L^p(\Omega\times\Omega)$ such that 
\begin{align}\label{func-G}
\chi_{(\Omega\times\Omega)\setminus (\Omega\times\Omega)_k(u_k)}U_{k,(p,s)}\rightharpoonup G\;\mbox{ in }\;  L^p(\Omega\times\Omega)\;\mbox{ as }\; k\to\infty.
\end{align}

Let $K_{k,p'}, K_{p'}:\Omega\times\Omega\to\RR$ be given by
\begin{align*}
K_{k,p'}(x,y):=\frac{\kappa_k(x-y)}{|x-y|^{\frac{N}{p'}+s-1}}\;\mbox{ and } K_{p'}(x,y):=\frac{\kappa(x-y)}{|x-y|^{\frac{N}{p'}+s-1}}.
\end{align*}
Since $\kappa_k\in L^\infty(\Omega)$, we have that
\begin{align*}
\int_{\Omega}\int_{\Omega}|K_{k,p'}(x,y)|^{p'}\;dxdy\le \|\kappa_k\|_{L^\infty(\Omega)}^{p'}\int_{\Omega}\int_{\Omega}
\frac{1}{|x-y|^{N+sp'-p'}}\;dxdy<\infty.
\end{align*}
Proceeding similarly we get that $K_{k,p'}, K_{p'}\in L^{p'}(\Omega\times\Omega)$.
Using the Lebesgue Dominated Convergence Theorem, we get that $K_{k,p'}\to K_{p'}$  in $L^{p'}(\Omega\times\Omega)$ as $k\to\infty$. Let $\varphi\in \mathcal D(\Omega)$ and define $\Phi:\Omega\times\Omega\to\RR$ by
$\displaystyle
\Phi(x-y):=\frac{\varphi(x)-\varphi(y)}{|x-y|}$.
Then, for every $\varphi\in \mathcal D(\Omega)$ we have that
\begin{align}\label{mj0}
&\lim_{k\to\infty}\int_{(\Omega\times\Omega)\setminus (\Omega\times\Omega)_k(u_k)}U_{k,(p,s)}(x,y)\Phi(x,y)K_{k,p'}(x-y)\;dxdy \notag\\
=&\int_{\Omega}\int_{\Omega}G(x,y)K_{p'}(x,y)\frac{\varphi(x)-\varphi(y)}{|x-y|}\;dxdy,
\end{align}
where $G$ is the function mentioned in \eqref{func-G}. 
We notice that for every $\varphi\in \mathcal D(\Omega)$,
\begin{align}\label{mj00}
&\int_{\Omega}\int_{\Omega}\frac{(u(x)-u(y))(\varphi(x)-\varphi(y))}{|x-y|^{N+2s}}\kappa(x-y)\;dxdy\notag\\
=&\lim_{k\to\infty}\int_{(\Omega\times\Omega)\setminus (\Omega\times\Omega)_k(u_k)}\frac{(u_{k}(x)-u_{k}(y))(\varphi(x)-\varphi(y))}{|x-y|^{N+2s}}\kappa_k(x-y)\;dxdy\notag\\
&+\lim_{k\to\infty}\int_{(\Omega\times\Omega)_k(u_k)}\frac{(u_{k}(x)-u_{k}(y))(\varphi(x)-\varphi(y))}{|x-y|^{N+2s}}\kappa_k(x-y)\;dxdy.
\end{align}
Using \eqref{AB6}, \eqref{eq-48}, \eqref{B2} and the H\"older inequality
 we get that there exists a constant $C>0$ (depending only on $\Omega,N,p$ and $s$) such that
\begin{align}\label{mj1}
&\left|\int\int_{(\Omega\times\Omega)_k(u_k)}\kappa_k(x-y)\frac{(u_{k}(x)-u_{k}(y))(\varphi(x)-\varphi(y))}{|x-y|^{N+2s}}\;dxdy\right|\notag\\
\le &C\|\varphi\|_{C^1(\bOm)}\|\kappa_k\|_{L^\infty(\Omega)}\left(\int\int_{(\Omega\times\Omega)_k(u_k)}\frac{1}{|x-y|^{N+2s-2}}\;dxdy\right)^{\frac 12}\notag\\
&\times\left(\int\int_{(\Omega\times\Omega)_k(u_k)}\frac{|u_{k}(x)-u_{k}(y)|^2}{|x-y|^{N+2s}}\;dxdy\right)^{\frac 12}\notag\\
\le &C\|\xi_2\|_{L^\infty(\Omega)}\|\varphi\|_{C^1(\bOm)}\frac{1}{(\varepsilon_k+n_k^2+1)^{\frac{p-2}{4}}}\sqrt{\mu\Big((\Omega\times\Omega)_k(u_k)\Big)}\|u_k\|_{\varepsilon_k,n_k,\kappa_k,s}^{\frac{p}{2}}\notag\\
\le &C\|\xi_2\|_{L^\infty(\Omega)}\|\varphi\|_{C^1(\bOm)}\frac{1}{n_k^{p-1}}\;\longrightarrow 0\;\mbox{ as }\;k\to\infty.
\end{align}
Using \eqref{mj1} we get from \eqref{mj0} and \eqref{mj00} that for every $\varphi\in\mathcal D(\Omega)$,
\begin{align*}
&\int_{\Omega}\int_{\Omega}G(x,y)K_{p'}(x,y)\frac{\varphi(x)-\varphi(y)}{|x-y|}\;dxdy\\
=&\int_{\Omega}\int_{\Omega}\kappa(x-y)\frac{u(x)-u(y)}{|x-y|^{\frac{N}{p}+s}}\frac{\varphi(x)-\varphi(y)}{|x-y|^{\frac{N}{p'}+s}}\;dxdy.
\end{align*}
Thus
$\displaystyle
G(x,y)=U_{(p,s)}(x,y)=\frac{u(x)-u(y)}{|x-y|^{\frac{N}{p}+s}}\;\mbox{ for a.e. }\; (x,y)\in\Omega\times\Omega$.
We have shown that
\begin{align}\label{wek-c}
\chi_{(\Omega\times\Omega)\setminus (\Omega\times\Omega)_k(u_k)}U_{k,(p,s)}\rightharpoonup U_{(p,s)}\;\mbox{ in }\;L^p(\Omega\times\Omega)\;\mbox{ as } k\to\infty.
\end{align}
Using \eqref{wek-c} and the fact that $u_k$ is a solution of \eqref{ellip-ep} we get that for every $k\in\NN$,
\begin{align}\label{ener-1}
&\frac{C_{N,p,s}}{2}\int_{\Omega}\int_{\Omega}\kappa_k(x-y)\Big[\varepsilon_k +\mathcal G_{n_k}\left(u_{k},s\right)\Big]^{\frac{p-2}{2}}\frac{|u_{k}(x)-u_{k}(y)|^2}{|x-y|^{N+2s}}\;dxdy\notag\\
&+\int_{\Omega}|u_k|^2\;dx=\int_{\Omega}fu_k\;dx
\end{align}
and
\begin{align}\label{ener-2}
\frac{C_{N,p,s}}{2}\int_{\Omega}\int_{\Omega}\kappa(x-y)\frac{|u(x)-u(y)|^p}{|x-y|^{N+sp}}\;dxdy+\int_{\Omega}|u|^2\;dx=\int_{\Omega}fu\;dx.
\end{align}
It follows from \eqref{ener-1}, \eqref{ener-2} and the fact that $u_k\rightharpoonup u$ in $W_0^{t,2}(\Omega)$ (as $k\to\infty$) that
\begin{align}
&\lim_{k\to\infty}\frac{C_{N,p,s}}{2}\int_{\Omega}\int_{\Omega}\kappa_k(x-y)\Big[\varepsilon_k +\mathcal G_{n_k}\left(u_{k},s\right)\Big]^{\frac{p-2}{2}}\frac{|u_{k}(x)-u_{k}(y)|^2}{|x-y|^{N+2s}}\;dxdy\notag\\
=&\int_{\Omega}fu\;dx-\int_{\Omega}|u|^2\;dx=\frac{C_{N,p,s}}{2}\int_{\Omega}\int_{\Omega}\kappa(x-y)\frac{|u(x)-u(y)|^p}{|x-y|^{N+sp}}\;dxdy.
\end{align}
Moreover,
\begin{align}\label{weak}
&\int_{\Omega}\int_{\Omega}\kappa(x-y)\frac{|u(x)-u(y)|^p}{|x-y|^{N+sp}}\;dxdy\notag\\
=&\lim_{k\to\infty}\int_{\Omega}\int_{\Omega}\kappa_k(x-y)\Big[\varepsilon_k +\mathcal G_{n_k}\left(u_{k},s\right)\Big]^{\frac{p-2}{2}}\frac{|u_{k}(x)-u_{k}(y)|^2}{|x-y|^{N+2s}}\;dxdy\notag\\
\ge &\limsup_{k\to\infty}\int_{\Omega}\int_{\Omega}\chi_{(\Omega\times\Omega)\setminus (\Omega\times\Omega)_k(u_k)}\kappa_k(x-y)\frac{|u_{k}(x)-u_{k}(y)|^p}{|x-y|^{{N+sp}}}\;dxdy\notag\\
\ge &\liminf_{k\to\infty}\int_{\Omega}\int_{\Omega}\chi_{(\Omega\times\Omega)\setminus (\Omega\times\Omega)_k(u_k)}\kappa_k(x-y)\frac{|u_{k}(x)-u_{k}(y)|^p}{|x-y|^{{N+sp}}}\;dxdy.
\end{align}
Since $\{\kappa_k\}_{k\in\NN}$ is bounded in $L^\infty(\Omega)$ and $\kappa_k(x)\ge \alpha$ for a.e. $x\in\Omega$ and every $k\in\NN$, then using \eqref{EQ} we get that 
\begin{align}\label{wea-c}
\chi_{(\Omega\times\Omega)\setminus (\Omega\times\Omega)_k(u_k)}U_{k,(p,s)}\widetilde\kappa_k^{\frac 1p}\rightharpoonup U_{(p,s)}\widetilde\kappa^{\frac 1p}\;\mbox{  in }\; L^p(\Omega\times\Omega)\;\mbox{ as }\;k\to\infty,
\end{align}
where $\widetilde\kappa_k(x,y)=\kappa_k(x-y)$ and $\widetilde\kappa=\kappa(x-y)$.
Using \eqref{weak} we get that
\begin{align*}
&\int_{\Omega}\int_{\Omega}\kappa(x-y)\frac{|u_k(x)-u_k(y)|^p}{|x-y|^{N+sp}}\;dxdy\notag\\
\ge &\limsup_{k\to\infty}\int_{\Omega}\int_{\Omega}\kappa_k(x-y)\chi_{(\Omega\times\Omega)\setminus (\Omega\times\Omega)_k(u_k)}\frac{|u_k(x)-u_k(y)|^p}{|x-y|^{N+sp}}\;dxdy\notag\\
\ge &\liminf_{k\to\infty}\int_{\Omega}\int_{\Omega}\kappa_k(x-y)\chi_{(\Omega\times\Omega)\setminus (\Omega\times\Omega)_k(u_k)}\frac{|u_k(x)-u_k(y)|^p}{|x-y|^{N+sp}}\;dxdy\notag\\
= &\liminf_{k\to\infty}\left\|\chi_{(\Omega\times\Omega)\setminus (\Omega\times\Omega)_k(u_k)}U_{k,p}\widetilde\kappa_k^{\frac 1p}\right\|_{L^p(\Omega\times\Omega)}^p\\
=&\|U_p\widetilde\kappa^{\frac 1p}\|_{L^p(\Omega\times\Omega)}^p=\int_{\Omega}\int_{\Omega}\kappa(x-y)\frac{|u_k(x)-u_k(y)|^p}{|x-y|^{N+sp}}\;dxdy.
\end{align*}
We have shown that
\begin{align}\label{wea-c2}
\lim_{k\to\infty}\left\|\chi_{(\Omega\times\Omega)\setminus (\Omega\times\Omega)_k(u_k)}U_{k,(p,s)}\widetilde\kappa_k^{\frac 1p}\right\|_{L^p(\Omega\times\Omega)}^p=\|U_{(p,s)}\widetilde\kappa^{\frac 1p}\|_{L^p(\Omega\times\Omega)}^p.
\end{align}
The weak convergence \eqref{wea-c} and the norm convergence \eqref{wea-c2} imply the strong convergence. From this we get that $\chi_{(\Omega\times\Omega)\setminus (\Omega\times\Omega)_k(u_k)}U_{k,(p,s)}\to U_{(p,s)}$ in $L^p(\Omega\times\Omega)$ as $k\to\infty$ and we have shown \eqref{58}.

{\bf Step 3}. Now we show \eqref{59}. Notice that it follows from \eqref{58} and \eqref{weak} that
\begin{align}\label{eq-14}
\lim_{k\to\infty}\int\int_{ (\Omega\times\Omega)_k(u_k)}\kappa_k(x-y)\Big[\varepsilon_k +\mathcal G_{n_k}\left(u_{k},s\right)\Big]^{\frac{p-2}{2}}\frac{|u_{k}(x)-u_{k}(y)|^2}{|x-y|^{N+2s}}\;dxdy=0.
\end{align}
Next we claim that
\begin{align}\label{eq-151}
&\lim_{k\to\infty}\int\int_{(\Omega\times\Omega)\setminus (\Omega\times\Omega)_k(u_k)}\kappa_k(x-y)\Big[\varepsilon_k +\mathcal G_{n_k}\left(u_{k},s\right)\Big]^{\frac{p-2}{2}}\frac{|u_{k}(x)-u_{k}(y)|^2}{|x-y|^{N+2s}}\;dxdy\notag\\
=&\int_{\Omega}\int_{\Omega}\kappa(x-y)\frac{|u_k(x)-u_k(y)|^p}{|x-y|^{N+sp}}\;dxdy.
\end{align}
Indeed, it follows from \eqref{eq-F} that
\begin{align}\label{eq-25}
&\left[\varepsilon_k +\mathcal F_{n_k}\left(\frac{|u_{k}(x)-u_{k}(y)|^2}{|x-y|^{2s}}\right)\right]^{\frac{p-2}{2}}\frac{|u_{k}(x)-u_{k}(y)|^2}{|x-y|^{N+2s}}\chi_{(\Omega\times\Omega)\setminus (\Omega\times\Omega)_k(u_k)}\notag\\
\le &\left[\varepsilon_k +\delta+\left(\frac{|u_{k}(x)-u_{k}(y)|^2}{|x-y|^{2s}}\right)\right]^{\frac{p-2}{2}}\frac{|u_{k}(x)-u_{k}(y)|^2}{|x-y|^{N+2s}}\chi_{(\Omega\times\Omega)\setminus (\Omega\times\Omega)_k(u_k)}\notag\\
\le &2^{\frac{p-2}{2}}\left[\left(\varepsilon_k +\delta\right)^{\frac{p-2}{2}}\frac{|u_{k}(x)-u_{k}(y)|^2}{|x-y|^{2s}}+\frac{|u_{k}(x)-u_{k}(y)|^p}{|x-y|^{N+sp}}\right]\chi_{(\Omega\times\Omega)\setminus (\Omega\times\Omega)_k(u_k)}.
\end{align}
Using \eqref{58},  \eqref{eq-25} and the Lebesgue Dominated Convergence Theorem we get \eqref{eq-151} and the claim is proved. Now \eqref{59} follows from \eqref{eq-14} and \eqref{eq-151}.

{\bf Step 4}.  It remains to show that $u_k\to u$ in $W_0^{t,2}(\Omega)$ as $k\to\infty$. Recall that by Step 1, $u_k\rightharpoonup u$  in $W_0^{t,2}(\Omega)$ as $k\to\infty$. Applying \eqref{eq-14} we can deduce that
\begin{align}\label{eq-26}
&\lim_{k\to\infty}\int\int_{ (\Omega\times\Omega)_k(u_k)}\frac{|u_k(x)-u_k(y)|^2}{|x-y|^{N+2t}}\;dxdy\\
\le &\frac{1}{\alpha}\lim_{k\to\infty}\int\int_{ (\Omega\times\Omega)_k(u_k)}\kappa_k(x-y)\Big[\varepsilon_k +\mathcal G_{n_k}\left(u_{k},s\right)\Big]^{\frac{p-2}{2}}\frac{|u_{k}(x)-u_{k}(y)|^2}{|x-y|^{N+2s}}\;dxdy=0.\notag
\end{align}
Now combining \eqref{eq-26} and \eqref{58} we get that
\begin{align*}
U_{k,(2,t)}=\chi_{(\Omega\times\Omega)_k(u_k)}U_{k,(2,t)}+\chi_{(\Omega\times\Omega)\setminus (\Omega\times\Omega)_k(u_k)}U_{k,(2,t)}\longrightarrow U_{(2,t)}\;\mbox{ in }\; L^2(\Omega\times\Omega),
\end{align*}
as $k\to\infty$. By Remark \ref{rem-1}, this implies \eqref{57}. The proof of the lemma  is finished.
\end{proof}

Now we are ready to give the proof of our last main result.

\begin{proof}[\bf Proof of Theorem \ref{theorem-23}]
It follows from Lemma \ref{lem-51} and the fact that the set $\mathfrak A_{ad}$ is bounded in $BV(\Omega)$ that  $\{(\kappa_{\varepsilon,n}^\star,u_{\varepsilon,n}^\star\}_{\varepsilon>0,n\in\NN}$ is bounded in $BV(\Omega)\times W_0^{t,2}(\Omega)$. Hence, after a (sub)sequence if necessary, $\kappa_{\varepsilon,n}^\star\weak u_\star$ in $BV(\Omega)$,  $u_{\varepsilon,n}^\star\rightharpoonup u_\star$ in $W_0^{t,2}(\Omega)$ and $u_{\varepsilon,n}^\star\to u_\star$ in $L^2(\Omega)$,  as $(\varepsilon,n)\to (0,\infty)$. It follows from Remark \ref{rem-21} that
\begin{align}\label{eq-21}
\lim_{(\varepsilon,n)\to(0,\infty)}\kappa_{\varepsilon,n}^\star=\kappa_\star\;\mbox{  in }\; L^1(\Omega)\;\mbox{ and }\; \int_{\Omega}|\nabla u_\star|\le \liminf_{(\varepsilon,n)\to(0,\infty)}\int_{\Omega}|\nabla u_{\varepsilon,n}^\star|.
\end{align}
This implies that $\kappa_\star\in \mathfrak A_{ad}$. It also follows from Lemmas \ref{lem-51} and \ref{lem-52} that $u_\star$ is a weak solution to  \eqref{ellip-eq} with $\kappa=\kappa_\star$. Thus $(\kappa_\star,u_\star)\in\Xi$.  Now combining \eqref{57} and \eqref{eq-21} we get  \eqref{210}. The convergences in \eqref{212} and \eqref{213} follow from \eqref{58} and \eqref{59}, respectively. Next we claim that $(\kappa_\star,u_\star)$ is a solution of \eqref{eq-min}. Given $(\kappa,u)\in\Xi$, $n\in\NN$ and $\varepsilon>0$, we define $\kappa_{\varepsilon,n}=\kappa$ and $u_{\varepsilon,n}$ as the solution of \eqref{ellip-ep}. Hence, $(\kappa_{\varepsilon,n},u_{\varepsilon,n})\in\Xi_{\varepsilon,n}$. It follows from \eqref{57} and \eqref{59} that
\begin{align}\label{eq-22}
\mathbb I(\kappa,u)=\lim_{(\varepsilon,n)\to(0,\infty)}\mathbb I(\kappa,u_{\varepsilon,n})=\lim_{(\varepsilon,n)\to(0,\infty)}\mathbb I(\kappa_{\varepsilon,n},u_{\varepsilon,n}).
\end{align}
Now using \eqref{57}, \eqref{210}, \eqref{eq-21}, \eqref{eq-22} and the fact that  $u_{\varepsilon,n}$ is the solution of \eqref{ellip-ep}, we get that
\begin{align*}
\mathbb I(\kappa_\star,u_\star)\le \liminf_{(\varepsilon,n)\to(0,\infty)}\mathbb I(\kappa_{\varepsilon,n}^\star,u_{\varepsilon,n}^\star)\le&\limsup_{(\varepsilon,n)\to(0,\infty)}\mathbb I(\kappa_{\varepsilon,n}^\star,u_{\varepsilon,n}^\star)\\
\le& \limsup_{(\varepsilon,n)\to(0,\infty)}\mathbb I(\kappa_{\varepsilon,n},u_{\varepsilon,n})=\mathbb I(\kappa,u).
\end{align*}
Since $(\kappa,u)$ is arbitrary in $\Xi$, it follows that $(\kappa_\star,u_\star)$ is a solution of \eqref{eq-min}. Moreover, taking $(\kappa,u)=(\kappa_\star,u_\star)$ in the above inequality we get \eqref{214}. Finally \eqref{211} follows directly from \eqref{214} and the fact that $\kappa_{\varepsilon,n}\weak u_\star$ in $BV(\Omega)$ and $u_{\varepsilon,n}\to u_\star$ in $L^2(\Omega)$, as $(\varepsilon,n)\to (0,\infty)$. The proof of the theorem is finished.
\end{proof}

\section{Optimal control of fractional $p$-Laplacian for $0<s<1$}\label{s:ocpF}

We conclude the article by mentioning that all our results are also valid if one replaces $\mathcal L_{\Omega,p}^s(\kappa,\cdot)$ with  $(-\Delta)_{p}^s(\kappa,\cdot)$ given in \eqref{op-LR}
with $0<s<1$.
In that case one replaces the state system \eqref{ellip-eq} by \eqref{e-frac-intro}
and $W_0^{s,p}(\Omega)$ by the space
\begin{align*}
W_0^{s,p}(\bOm):=\Big\{u\in W^{s,p}(\RR^N):\;u=0\;\mbox{ on }\;\RR^N\setminus\Omega\Big\}.
\end{align*}
Under Assumption \ref{asump}(a), it has been shown in \cite[Theorem 6]{Val} that $\mathcal D(\Omega)$ is dense in $W_0^{s,p}(\bOm)$. Moreover, for every $0<s<1$,
\begin{align*}
\|u\|_{W_0^{s,p}(\bOm)}:=\left(\int_{\RR^N}\int_{\RR^N}\frac{|u(x)-u(y)|^p}{|x-y|^{N+sp}}\;dxdy\right)^{\frac 1p}
\end{align*}
defines an equivalent norm on $W_0^{s,p}(\bOm)$. In addition we have that $W_0^{s,p}(\bOm)=W_0^{s,p}(\Omega)$ with equivalent norms if $\frac 1p<s<1$ (see e.g. \cite[Section 1.1]{AW2017note}). 
With the above setting, a weak solution of \eqref{e-frac-intro} is defined to be a function $u\in W_0^{s,p}(\bOm)$ such that for every $\varphi\in W_0^{s,p}(\bOm)$ the equality
\begin{align*}
&\frac{C_{N,p,s}}{2}\int_{\RR^N}\int_{\RR^N}\kappa(x-y)|u(x)-u(y)|^{p-2}\frac{(u(x)-u(y))(\varphi(x)-\varphi(y))}{|x-y|^{N+sp}}\;dxdy\notag\\
& +\int_{\Omega}u\varphi\;dx=:\widetilde{\mathcal E}_{p,s}^\kappa(u,\varphi)=\int_{\Omega}f\varphi\;dx,
\end{align*}
holds. The associated regularized problem is given by
\begin{equation*}
\begin{cases}
(-\Delta)_{p,\varepsilon,n}^s(\kappa,u)+u=f\;\;&\mbox{ in }\;\Omega\\
u=0&\mbox{ on }\;\RR^N\setminus\Omega,
\end{cases}
\end{equation*}
with
\begin{align*}
(-\Delta)_{p,\varepsilon,n}^s(\kappa,u):=
C_{N,p,s}\mbox{P.V.}\int_{\RR^N}\kappa(x-y)\left[\varepsilon+\mathcal G_n\left(u,s\right)\right]^{\frac{p-2}{2}}\frac{u(x)-u(y)}{|x-y|^{N+2s}}\;dy,
\end{align*}
and a weak solution is defined to be a $u\in W_0^{s,2}(\bOm)$ such that for every $\varphi\in W_0^{s,2}(\bOm)$,
\begin{align*}
&\frac{C_{N,p,s}}{2}\int_{\RR^N}\int_{\RR^N}\kappa(x-y)\left[\varepsilon+\mathcal G_n\left(u,s\right)\right]^{\frac{p-2}{2}}\frac{(u(x)-u(y))(\varphi(x)-\varphi(y))}{|x-y|^{N+2s}}\;dxdy\notag\\
& +\int_{\Omega}u\varphi\;dx=:\widetilde{\mathbb F}_{\varepsilon,n,p}^\kappa(u,\varphi)=\int_{\Omega}f\varphi\;dx.
\end{align*}
All our results hold with very minor changes in the proofs, if one replaces the expressions of $\mathcal E_{p,s}^\kappa$ and $\mathbb F_{\varepsilon,n,p}^\kappa$ given in \eqref{form} and \eqref{func-FF}, respectively, by
$\widetilde{\mathcal E}_{p,s}^\kappa$
and $\widetilde{\mathbb F}_{\varepsilon,n,p}^\kappa(u,\varphi)$
for $u,\varphi\in W_0^{s,2}(\bOm)$, respectively. In this case $\xi_1,\xi_2\in L^\infty(\RR^N)$ and
\begin{align*}
\kappa\in \widetilde{\mathfrak{A}}_{ad}:=\Big\{\eta\in BV(\Omega):\;0<\alpha\le \xi_1(x)\le \eta(x)\le \xi_2(x)\;\mbox{ a.e. in }\;\RR^N\Big\}.
\end{align*}
In addition in this situation, all the results holds for every $0<s<1$.

\bibliographystyle{siam}
\bibliography{biblio}

\end{document}